\documentclass[reqno,11pt]{amsart}

\usepackage{amsmath, amsthm, amssymb, amsfonts, amsthm}
\usepackage{amsxtra, amscd, geometry, graphicx,epic,eepic}
\usepackage[all,cmtip]{xy}
\usepackage{mathrsfs}
\usepackage{hyperref}

\newtheorem{theorem}{Theorem}         % 1st argument is your name for it
\newtheorem{proposition}{Proposition}%[section] % 2nd argument is what is printed
\newtheorem{cor}[proposition]{Corollary}
\newtheorem{lemma}[proposition]{Lemma}

\newtheorem{definition}[proposition]{Definition}

\theoremstyle{definition}

\newtheorem*{rmk*}{Remark}
\newcommand{\XX}{\mathcal X}
\newcommand{\YY}{\mathcal Y}
\newcommand{\ZZ}{\mathcal Z}
\newcommand{\N}{\mathbb N}
\newcommand{\Q}{\mathbb Q}
\newcommand{\Z}{\mathbb Z}
\newcommand{\C}{\mathbb C}
\newcommand{\DD}{\mathcal D}
\newcommand{\wT}{\widetilde{T}}

\title[Minkowski type question mark functions]{On Minkowski type question mark functions associated with even or odd continued fractions}
\author{Florin P. Boca}
\author{Christopher Linden}

\address{Department of Mathematics, University of Illinois at Urbana-Champaign,
Urbana, IL 61801}

\address{E-mail: fboca@math.uiuc.edu \quad clinden2@illinois.edu}

\date{\today}

\begin{document}

\begin{abstract}
We study analogues of Minkowski's question mark function $?(x)$ related to continued fractions with even or with odd partial quotients.
We prove that these functions are H\"older continuous with precise exponents, and that they linearize the appropriate versions of the Gauss and Farey maps.
\end{abstract}

\maketitle

\section{Introduction}
Minkowski \cite{Min} introduced a homeomorphism of $[0,1]$, which he denoted $?(x)$, that gives monotonic bijections between rational
and dyadic numbers in $[0,1]$, and also between quadratic irrationals in $(0,1)$ and rationals in $(0,1)$.
The function $?$ is singular, yet strictly increasing, continuous, and surjective.
The question mark can be defined inductively on rationals by
\begin{equation*}
?\bigg( \frac{p+p^\prime}{q+q^\prime} \bigg) =\frac{1}{2}
  ? \bigg( \frac{p}{q}\bigg)+\frac{1}{2} ? \bigg(
  \frac{p^\prime}{q^\prime}\bigg),
  \end{equation*}
whenever $\frac{p}{q}$ and $\frac{p^\prime}{q^\prime}$ are rational numbers in lowest terms in $[0,1]$
with $p^\prime q-pq^\prime =1$.
In terms of the regular continued fraction expansion
\begin{equation*}
x=[a_1,a_2,\ldots]=\cfrac{1}{a_1+\frac{1}{a_2+\frac{1}{\ddots}}} ,
\end{equation*}
the values of $?(x)$ can be explicitly expressed by Denjoy's formula \cite{Den}
(see also \cite{Sal} and \cite{KMS}) as
\begin{equation*}
? ([a_1,a_2,a_3,\ldots ]) =\frac{1}{2^{a_1-1}} - \frac{1}{2^{a_1+a_2-1}}
 +\frac{1}{2^{a_1+a_2+a_3-1}}   -  \cdots   .
\end{equation*}

It is well-known (see, e.g., \cite{BI}) that $?(x)$ linearizes the
classical Gauss and Farey maps
\begin{equation*}
G \big( [a_1,a_2,a_3,\ldots]\big) =[a_2,a_3,a_4,\ldots ],\quad
F \big( [a_1,a_2,a_3,\ldots ]\big) =\begin{cases}
[a_1 -1,a_2,a_3,\ldots ] & \mbox{\rm if $a_1 \geq 2$} \\
[a_2,a_3,a_4,\ldots] & \mbox{\rm if $a_1=1$,} \end{cases}
\end{equation*}
associated with regular continued fractions, or equivalently
\begin{equation*}
G(x)=\left\{ \frac{1}{x}\right\} =\frac{1}{x}
-\left[ \frac{1}{x}\right],\quad
F(x)=\begin{cases}
\frac{x}{1-x} & \mbox{\rm if $x\in \left[ 0,\frac{1}{2}\right]$} \\
\frac{1-x}{x} & \mbox{\rm if $x\in \left[ \frac{1}{2},1\right] .$} \end{cases}
\end{equation*}
More precisely, the map $? G?^{-1}$ is decreasing and is linear on each interval
$(2^{-k-1},2^{-k})$, while $(? F?^{-1})(x)=2\mbox{\rm dist} (x,\Z)$.

Salem \cite{Sal} proved that $?(x)$ is singular and H\"older continuous, with best exponent $\frac{\log 2}{2\log G}\approx 0.72021$, where
$G=\frac{1}{2} (1+\sqrt{5})$ denotes the ``big" golden ratio.
Several significant results about $?(x)$ have subsequently been proved \cite{Kin,PVB,Alk,DKM,DVA}, culminating with the
very recent solution provided by Jordan and Sahlsten \cite{JS}
to the longstanding Salem open problem \cite{Sal} concerning the decay of its Fourier-Stieltjes coefficients.
A number of generalizations of this classical map have been considered
\cite{GKT,BG,Pan,Zha,Nor}.
See \url{http://uosis.mif.vu.lt/~alkauskas/minkowski.htm} for an extensive bibliography
of research in this area until 2014.

Any continued fraction algorithm on $[0,1]$ generates a natural filtration $\{ \YY_n\}$ of
$\Q \cap [0,1]$, obtained by taking into account the sum of the partial quotients of
the rationals. One can consider the simple non-decreasing functions
\begin{equation*}
Q_n :[0,1] \rightarrow [0,1],\quad
Q_n (x):= \frac{\lvert \{ y\in \YY_n : y<x \} \vert}{\lvert \YY_n \rvert -1} .
\end{equation*}
For the regular continued fraction, $\YY_n$ is the set of rationals with sum of partial
quotients at most $n$. The limit $Q(x):=\lim_n Q_n (x)$ provides an analogue of the Minkowski question mark function.

This paper is concerned with the study of the resulting maps $Q_E$ and $Q_O$ in the situation of
continued fractions with even or with odd partial quotients,
that we are simply going to call \emph{even continued fractions} (in short ECF), and respectively \emph{odd continued
fractions} (in short OCF). See \cite{Sch1,Sch2} for the definition and basic properties of these two classes of continued fractions and \cite{Rie} for a
detailed treatment of odd continued fractions. The set $\YY_n$ is defined in \eqref{eq2} and \eqref{eq3} for the ECF situation,
and in \eqref{eq7} for the OCF situation. The map $Q_O$ has been previously introduced and investigated
by Zhabitskaya \cite{Zha}.

In Section \ref{ECF} we consider the situation of even continued fractions, defining our even question mark function $Q_E$ and proving a formula for $Q_E(x)$ in terms of
the ECF expansion of $x$. As a consequence, we prove that $Q_E$ is singular and H\"older continuous with best exponent
$\frac{\log 3}{2 \log (1+ \sqrt{2})}\approx 0.62324$. We also show that $Q_E$ linearizes the even Gauss and even Farey maps. As the formula in Theorem \ref{Thm1} makes clear,
$Q_E$ is naturally a triadic version of Minkowski's $?(x)$ function. Northshield has introduced \cite{Nor} a different triadic generalization of the
question mark function. At the end of the section, we establish a precise connection between our even continued fraction analogue of the Stern sequence and the
sequence in $\mathbb{Z}[\sqrt{2}]$ that he considers.

In Section \ref{OCF} we focus on odd continued fractions, following Zhabitskaya's work \cite{Zha}
and considering the odd question mark function $Q_O(x)$ that coincides with her $F^0(x)$.
We prove that the function $Q_O$ is H\"older continuous with best exponent $\frac{\log\lambda}{2\log G}\approx 0.63317$,
where $\lambda\approx 1.83929$ denotes the unique real root of the equation $x^3-x^2-x-1=0$. We also prove that the map $Q_O$ linearizes the
odd Gauss and the odd Farey maps.

\begin{center}
\begin{figure}
\includegraphics[scale=0.7,bb = 0 0 280 180]{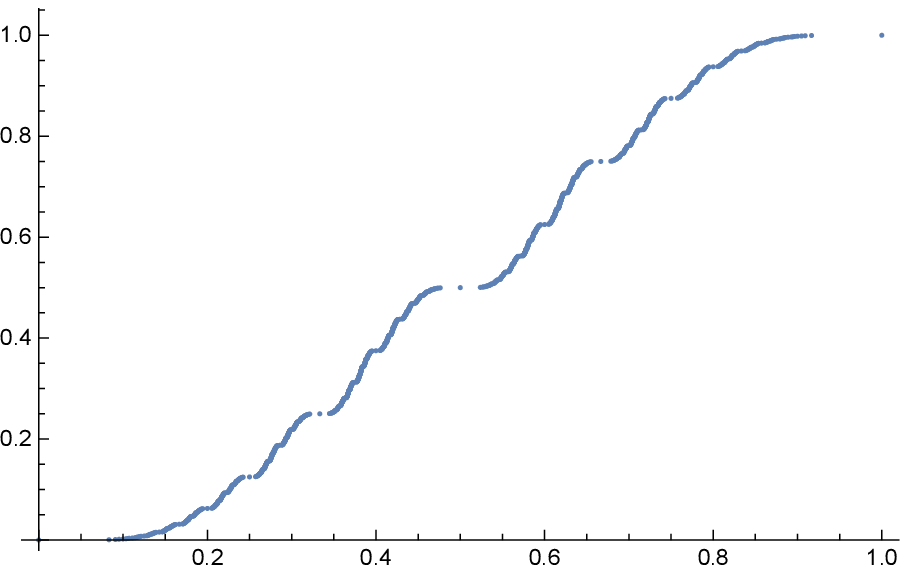} \\
\includegraphics[scale=0.7,bb = 0 0 280 180]{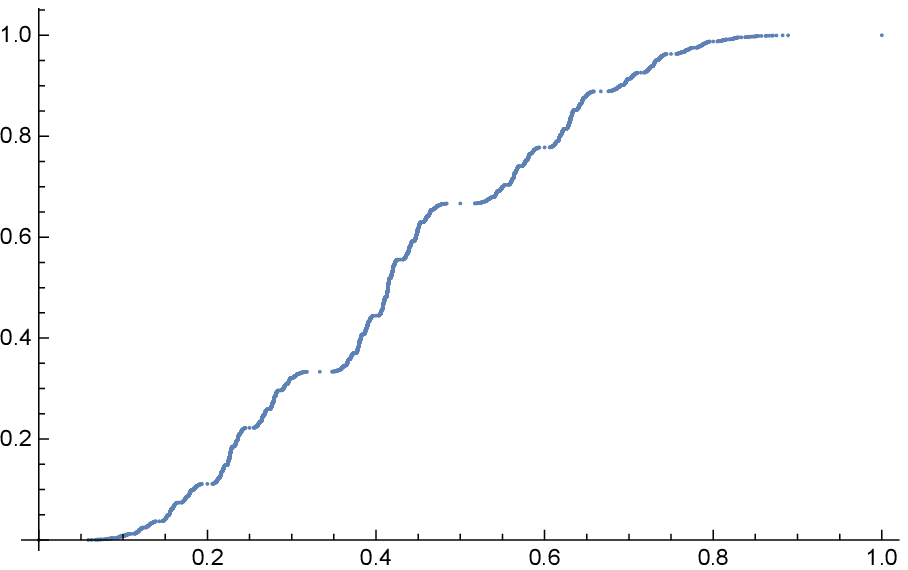}
\includegraphics[scale=0.7,bb = 0 0 280 180]{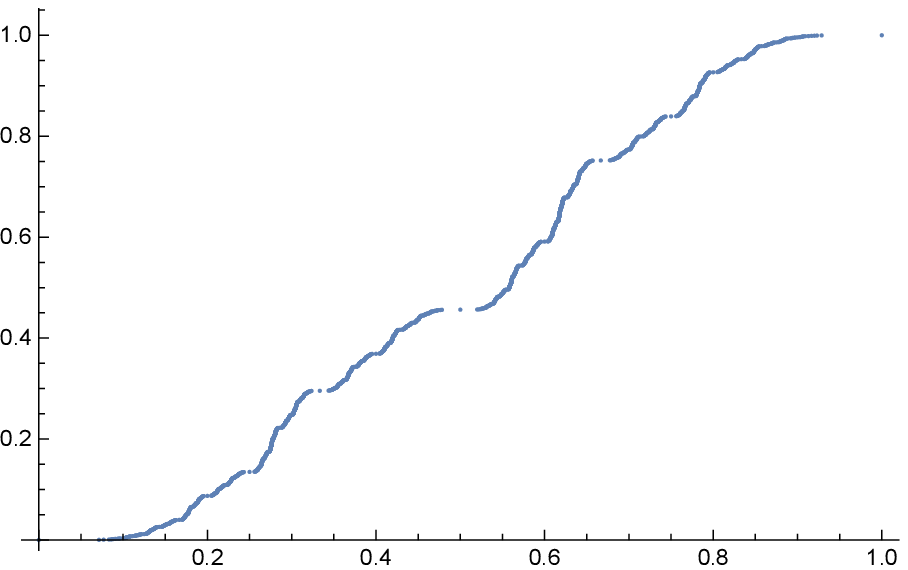}
\caption{The graphs of the functions $?(x)$, $Q_E(x)$ and $Q_O(x)$}\label{Figure1}
\end{figure}
\end{center}

\section{Even Partial Quotients}\label{ECF}
\subsection{Even continued fraction and associated Gauss and Farey maps}
We consider the ECF expansion in $[-1,1]$ given by
\begin{equation}\label{eq1}
[(e_1, a_1), (e_2, a_2), (e_3,a_3), \ldots ] = \cfrac{e_1}{a_1+\cfrac{e_2}{a_2+\frac{e_3}{\ddots}}},
\end{equation}
where $e_i \in \{\pm1\}$ and $a_i \in 2\mathbb{N}$.
Although most of the time we shall only consider positive numbers, that is $e_1=1$,
it will be sometimes convenient to consider the full range $\{ \pm 1\}$ for $e_1$, especially
when working with the function $Q_E$ or in Subsection 2.5.

For uniqueness, we shall require that in a finite expansion,
the last $e_j$ must equal $1$, and in this case we allow $a_j$ to also equal $1$.
This convention allows all
rational numbers to have a unique finite even continued fraction expansion.
For example, we have $\frac{1}{2k}=[(1,2k)]$, $\frac{1}{2k+1}= [(1,2k),(1,1)]$,
$\frac{3}{8}=[(1,2),(1,2),(-1,2)]$, $\frac{5}{13}=[(1,2),(1,2),(-1,2),(1,1)]$.
Note that $\frac{p}{q}$ will also have
a (unique) infinite expansion if and only if $p + q \equiv 0 \pmod 2$
and if and only if its finite expansion terminates in a $1$.

The corresponding Farey type map $F_E :[0,1]\rightarrow [0,1]$
defined by
\begin{equation*}
F_E(x)=\begin{cases}  \frac{x}{1-2x} & \mbox{\rm if $0\leq x\leq \frac{1}{3}$} \\
 \frac{1}{x}-2 & \mbox{\rm if $\frac{1}{3} \leq x \leq \frac{1}{2}$} \\
 2-\frac{1}{x} & \mbox{\rm if $\frac{1}{2} \leq x \leq 1$,}
\end{cases}
\end{equation*}
with infinite invariant measure $d\nu_E (x)=\frac{dx}{x(1-x)}$, has been already considered in different contexts in
\cite{AaD} and \cite{Rom}.
Symbolically, $F_E$ acts on the ECF representation by subtracting $2$ from the leading digit $a_1$ of $x$ when
$a_1 \geq 4$ (which corresponds to $x$ between $0$ and $\frac{1}{3}$), and by simply removing $(a_1,e_2)$ when $a_1=2$
(which corresponds to $x$ between $\frac{1}{3}$ and $1$), i.e.
\begin{equation*}
F_E ([(1,a_1),(e_2,a_2),(e_3,a_3),\ldots]) =\begin{cases}
[(1,a_1-2),(e_2,a_2),(e_3,a_3),\ldots ] & \mbox{\rm if $a_1 \geq 4$} \\
[(1,a_2),(e_3,a_3),(e_4,a_4),\ldots ] & \mbox{\rm if $a_1=2$.}
\end{cases}
\end{equation*}

\begin{center}
\begin{figure}
\includegraphics[scale=0.65,bb = 0 20 300 260]{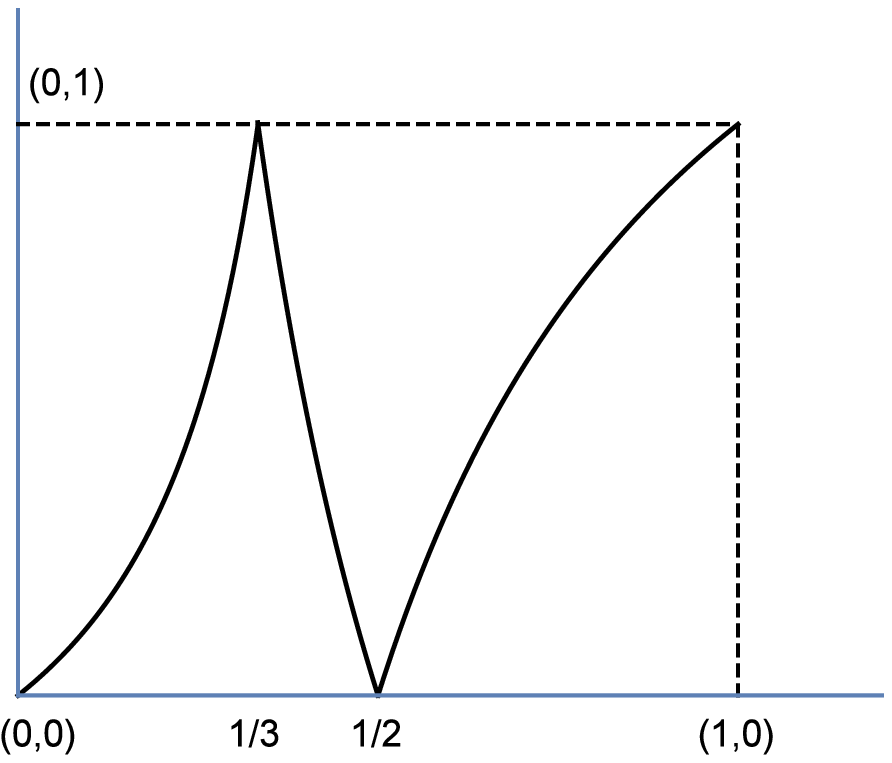}
\includegraphics[scale=0.65,bb = 0 20 200 260]{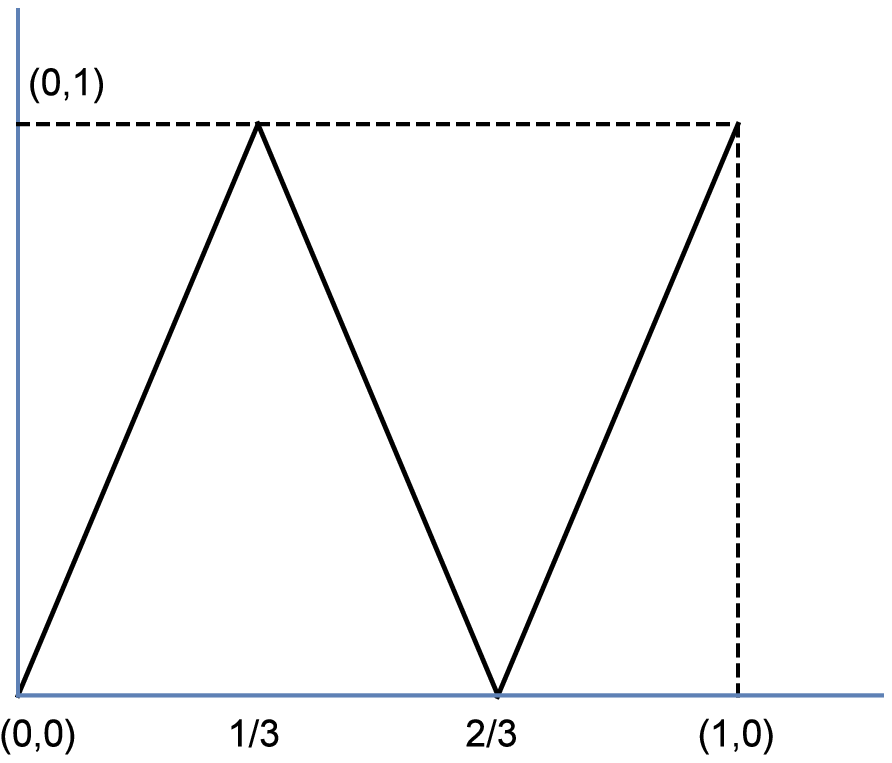}
\caption{The even Farey map $F_E$ and its linearization $\overline{F}_E$}\label{Figure2}
\end{figure}
\end{center}

We shall be interested in the sets $\YY_n$, $\ZZ_n$, and $\XX_n$ defined by
\begin{equation}\label{eq2}
\YY_n =F_E^{-n} (\{ 0,1\}), \quad
\ZZ_n =F_E^{-n} (\{ 0\})\quad  \mbox{\rm and} \quad
\XX_n := \YY_n \setminus \YY_{n-1},
\end{equation}
of cardinality $X_n:=\lvert \XX_n\rvert$, $Y_n:=\lvert \YY_n \rvert$ and $Z_n =\lvert \ZZ_n \rvert$.
Our convention is to take $0,1\in \YY_n$ and $0\in \ZZ_n$.
It is plain to check that $\ZZ_0 =\{ \frac{0}{1} \}$ and
\begin{equation*}
\begin{split}
\ZZ_1 & =\bigg\{ \frac{0}{1},\frac{1}{2}\bigg\} ,\quad
\ZZ_2 =\bigg\{ \frac{0}{1},\frac{1}{4},\frac{2}{5},\frac{1}{2},\frac{2}{3}\bigg\} ,\quad
\ZZ_3 =\bigg\{ \frac{0}{1},\frac{1}{6},\frac{2}{9},\frac{1}{4}, \frac{2}{7},\frac{3}{8}, \frac{2}{5},
\frac{5}{12},\frac{4}{9},\frac{1}{2}, \frac{4}{7}, \frac{5}{8}, \frac{2}{3}, \frac{3}{4} \bigg\} , \\
\YY_0 & =\bigg\{ \frac{0}{1},\frac{1}{1}\Big\}, \quad \YY_1 =\Big\{ \frac{0}{1},\frac{1}{3},\frac{1}{2},\frac{1}{1}\bigg\} ,
\quad \YY_2 =\bigg\{ \frac{0}{1},\frac{1}{5},\frac{1}{4},\frac{1}{3},\frac{2}{5},\frac{3}{7},
\frac{1}{2},\frac{3}{5},\frac{2}{3},\frac{1}{1}\bigg\}.
\end{split}
\end{equation*}

The first return map $R_E$ of $F_E$ on $(\frac{1}{3},1]$ acts on the ECF expansion as
\begin{equation*}
R_E ( [(1,2),(e_2,a_2),(e_3,a_3),\ldots ]) =[(1,2),(e_3,a_3),(e_4,a_4),\ldots] .
\end{equation*}
Recall that the even Gauss map $T_E$ acts on $[0,1]$ by $T_E (0)=0$ and
\begin{equation*}
T_E(x)= \bigg| \frac{1}{x}-2\bigg[ \frac{1}{2x}+\frac{1}{2}\bigg]\bigg|
= \bigg| \frac{1}{x}-2k \bigg| \quad \mbox{\rm if $\displaystyle x\in \bigg[ \frac{1}{2k+1},\frac{1}{2k-1}\bigg],$}
\end{equation*}
and it acts on ECF expansions \eqref{eq1} restricted to $(0,1)$ by
\begin{equation*}
T_E ( [(1,a_1),(e_2,a_2),(e_3,a_3),\ldots ]) =
[(1,a_2),(e_3,a_3),(e_4,a_4),\ldots ].
\end{equation*}
Furthermore, $d\mu_E (x)=(\frac{1}{1+x} + \frac{1}{1-x})dx$ is a $T_E$-invariant measure \cite{Sch1}.
Consider also the extended ECF Gauss map $\wT_E :[-1,1) \rightarrow [-1,1)$, acting on
the ECF expansion \eqref{eq1} as
\begin{equation*}
\wT_E ( [(e_1,a_1),(e_2,a_2),(e_3,a_3),\ldots ]) =[(e_2,a_2),(e_3,a_3),(e_4,a_4),\ldots ].
\end{equation*}
Equivalently, we can take $\wT_E (0)=0$ and
\begin{equation*}
\wT_E (x)= \frac{1}{\lvert x\rvert} -2\bigg[ \frac{1}{2\lvert x\rvert}+\frac{1}{2}\bigg] \quad \mbox{\rm if $x\neq 0$.}
\end{equation*}
The push-forward measure $d\widetilde{\mu}_E (x) =\frac{dx}{1+x}$ of $\nu_E\vert_{(1/3,1]}$ under $\varphi$ is $\wT_E$-invariant,
where $\varphi:(\frac{1}{3},1]\rightarrow [-1,1)$,
$\varphi (x)=\frac{1}{x}-2$ with $\varphi^{-1} (y)=\frac{1}{2+y}$.
It is plain that $R_E$ and $\wT_E$ are conjugated, and more precisely $\wT_E=\varphi R_E \varphi^{-1}$.
It is also plain that $\wT_E$ is an extension of $T_E$.
More precisely we have $\pi \wT_E=T_E \pi$, where
$\pi (x)=\lvert x\rvert$. The push forward of $\widetilde{\mu}_E$ under $\pi$ is
the $T_E$-invariant measure $\mu_E$.

\begin{center}
\begin{figure}
\includegraphics[scale=0.65,bb = 0 20 300 280]{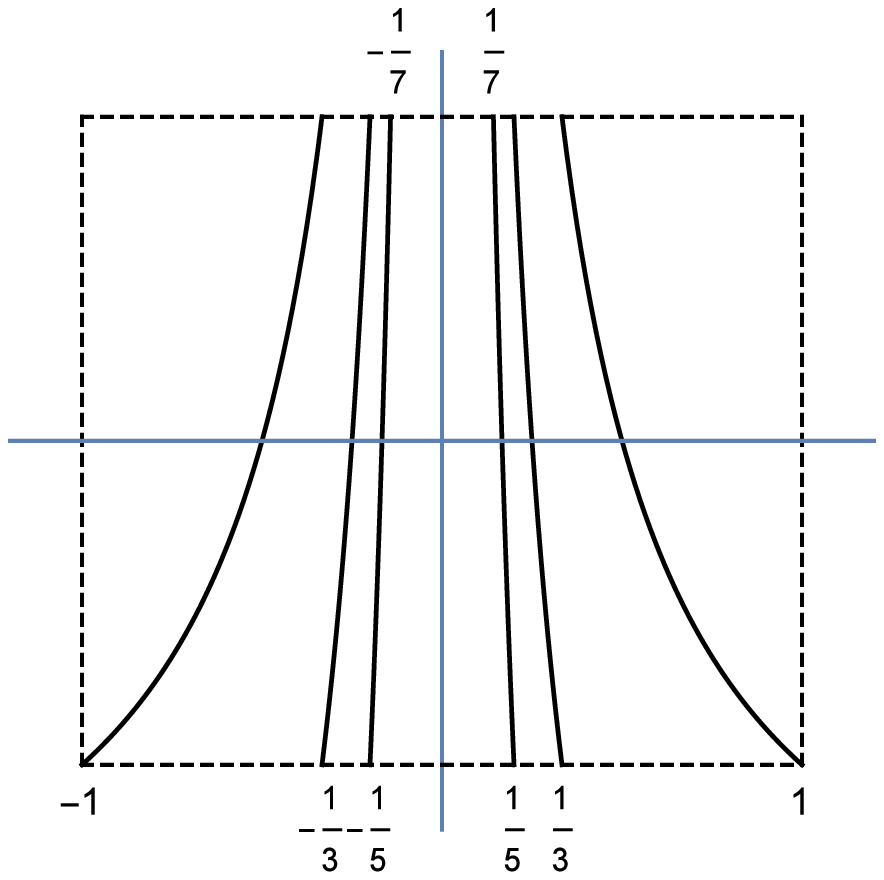}
\includegraphics[scale=0.65,bb = 0 20 200 280]{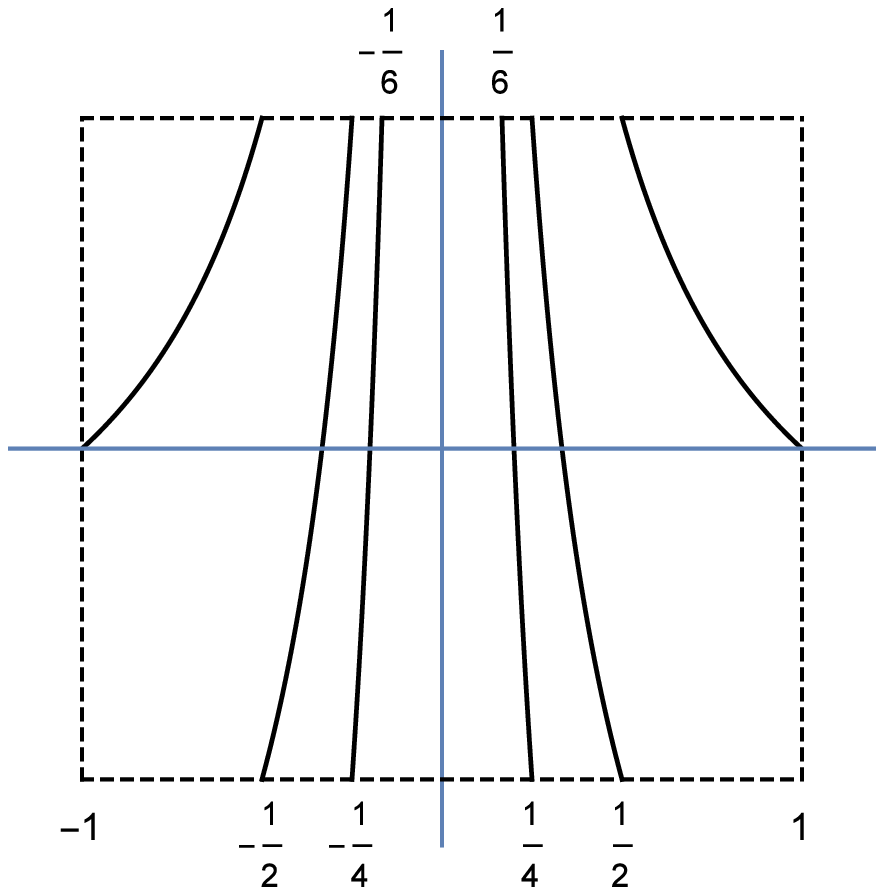}
\caption{The extended Gauss maps $\wT_E$ and $\wT_O$}\label{Figure3}
\end{figure}
\end{center}

\subsection{Ordering of rational numbers associated with the even continued fraction}
If $x = [(1, a_1), (e_2, a_2), \ldots ,(e_n, a_n)] \in \mathcal{Z}_k$, then let $[(x), (\epsilon_1, \alpha_1), (\epsilon_2, \alpha_2), \ldots ]$
denote the concatenated expansion  $[(1, a_1), (e_2, a_2), \ldots ,(e_n, a_n), (\epsilon_1, \alpha_1), (\epsilon_2, \alpha_2), \ldots ]$.
Observe that the sets $\YY_k$ and $\ZZ_k$ defined in \eqref{eq2} can also be described as
\begin{equation}\label{eq3}
\begin{split}
\YY_k = & \big\{ x= [(1, a_1), (e_2, a_2), \ldots ,(e_n, a_n)]\in \Q \cap [0,1]: a_1+\cdots+a_n \leq 2k+1 \big\} , \\
\ZZ_k = &  \{ x \in \YY_k: \forall i, a_i \neq 1 \}.
\end{split}
\end{equation}
Observe also that $\YY_k = \bigcup_{x \in \ZZ_k} \{ x, [(x), (1,1)]\}$, and hence $Y_k = 2Z_k$. For $x \in \ZZ_k$,
consider $c_k(x) := 2k+2 - \sum a_i$. Then
$$\ZZ_{k+1} = \left(\bigcup_{x \in \ZZ_k, x \neq 0} \{ x, [(x), (1, c_k(x))], [(x), (-1, c_k(x))] \} \right) \cup \{ 0, [(1,c_k (0))]\}.$$
Hence $Z_{k+1} = 3Z_k-1$. Since $Z_0 =1$, we have $Z_k = \frac{3^k+1}{2}$ and we conclude that $Y_k = 3^k +1$.

\begin{figure}%[htb]
\begin{center}
\unitlength 0.5mm
\begin{picture}(0,100)(0,33)
%\thinlines
%\dottedline{2}(-120,100)(-120,120)(120,100)

\path(-120,40)(-120,120)(120,100)(120,40)
\path(-120,100.5)(-40,80.5) \path(-120,99.5)(-40,79.5) \path(-40,80)(120,100)
\put(-40,80){\makebox(0,0){{\tiny $\bullet$}}}

\path(-120,100)(40,80)(120,100)
\put(40,80){\makebox(0,0){{\tiny $\bullet$}}}

\path(-40,80)(-40,40) \path(40,80)(40,40)
\put(40,60){\makebox(0,0){{\tiny $\bullet$}}}
\put(40,40){\makebox(0,0){{\tiny $\bullet$}}}
\put(-40,60){\makebox(0,0){{\tiny $\bullet$}}}
\put(-40,40){\makebox(0,0){{\tiny $\bullet$}}}

\put(-93.7,60){\makebox(0,0){{\tiny $\bullet$}}}
\put(-93.7,40){\makebox(0,0){{\tiny $\bullet$}}}
\put(-66.6,60){\makebox(0,0){{\tiny $\bullet$}}}
\put(-66.6,40){\makebox(0,0){{\tiny $\bullet$}}}

\put(93.7,60){\makebox(0,0){{\tiny $\bullet$}}}
\put(93.7,40){\makebox(0,0){{\tiny $\bullet$}}}
\put(66.6,60){\makebox(0,0){{\tiny $\bullet$}}}
\put(66.6,40){\makebox(0,0){{\tiny $\bullet$}}}

\put(-120,100){\makebox(0,0){{\tiny $\bullet$}}}
\put(-120,80){\makebox(0,0){{\tiny $\bullet$}}}
\put(-120,60){\makebox(0,0){{\tiny $\bullet$}}}
\put(-120,40){\makebox(0,0){{\tiny $\bullet$}}}
\put(120,100){\makebox(0,0){{\tiny $\bullet$}}}
\put(120,80){\makebox(0,0){{\tiny $\bullet$}}}
\put(120,60){\makebox(0,0){{\tiny $\bullet$}}}
\put(120,40){\makebox(0,0){{\tiny $\bullet$}}}

\path(-120,80.5)(-93.7,60.5) \path(-120,79.5)(-93.7,59.5) \path(-93.7,60)(-40,80)
\path(-120,80)(-66.6,60)(-40,80)

\put(-13.4,60){\makebox(0,0){{\tiny $\bullet$}}}
\put(-13.4,40){\makebox(0,0){{\tiny $\bullet$}}}
\put(13.4,60){\makebox(0,0){{\tiny $\bullet$}}}
\put(13.4,40){\makebox(0,0){{\tiny $\bullet$}}}

\path(-40,80)(-13.4,60)(40,80)(93.7,60)(120,80)
\path(13.4,60.5)(40,80.5) \path(13.4,59.5)(40,79.5) \path(-40,80)(13.4,60)
\path(120,80)(66.6,60) \path(66.6,60.5)(40,80.5) \path(66.6,59.5)(40,79.5)

\put(111.14,40){\makebox(0,0){{\tiny $\bullet$}}}
\put(102.56,40){\makebox(0,0){{\tiny $\bullet$}}}
\put(84.84,40){\makebox(0,0){{\tiny $\bullet$}}}
\put(75.46,40){\makebox(0,0){{\tiny $\bullet$}}}
\put(57.74,40){\makebox(0,0){{\tiny $\bullet$}}}
\put(48.86,40){\makebox(0,0){{\tiny $\bullet$}}}
\put(31.14,40){\makebox(0,0){{\tiny $\bullet$}}}
\put(22.26,40){\makebox(0,0){{\tiny $\bullet$}}}
\put(4.54,40){\makebox(0,0){{\tiny $\bullet$}}}
\put(-111.14,40){\makebox(0,0){{\tiny $\bullet$}}}
\put(-102.56,40){\makebox(0,0){{\tiny $\bullet$}}}
\put(-84.84,40){\makebox(0,0){{\tiny $\bullet$}}}
\put(-75.46,40){\makebox(0,0){{\tiny $\bullet$}}}
\put(-57.74,40){\makebox(0,0){{\tiny $\bullet$}}}
\put(-48.86,40){\makebox(0,0){{\tiny $\bullet$}}}
\put(-31.14,40){\makebox(0,0){{\tiny $\bullet$}}}
\put(-22.26,40){\makebox(0,0){{\tiny $\bullet$}}}
\put(-4.54,40){\makebox(0,0){{\tiny $\bullet$}}}

\path(93.7,40)(93.7,60) \path(66.6,40)(66.6,60)
\path(13.4,40)(13.4,60)
\path(-93.7,40)(-93.7,60) \path(-66.6,40)(-66.6,60)
\path(-13.4,40)(-13.4,60)

\path(-120,60)(-102.56,40)(-93.7,60)(-84.84,40)(-66.6,60)(-48.86,40)(-40,60)(-31.14,40)(-13.4,60)(4.54,40)(13.4,60)(22.26,40)(40,60)(57.74,40)(66.6,60)(75.46,40)(93.7,60)(111.14,40)(120,60)

\path(-120,61)(-111.14,41) \path(-120,59)(-111.14,39) \path(-111.14,40)(-93.7,60)(-84.84,40)
\path(-93.7,60)(-75.46,40) \path(-75.46,41)(-66.6,61)(-57.74,41) \path(-75.46,39)(-66.6,59)(-57.74,39)
\path(-57.74,40)(-40,60)(-22.26,40) \path(-22.26,41)(-13.4,61)(-4.54,41) \path(-22.26,39)(-13.4,59)(-4.54,39)
\path(-4.54,40)(13.4,60)(31.14,40) \path(31.14,41)(40,61)(48.84,41) \path(31.14,39)(40,59)(48.84,39)
\path(48.84,40)(66.6,60)(84.84,40) \path(84.84,41)(93.7,61)(102.56,41) \path(84.84,39)(93.7,59)(102.56,39)
\path(102.56,40)(120,60)

\put(-125,100){\makebox(0,0){{\small $\frac{0}{1}$}}}

\put(125,100){\makebox(0,0){{\small $\frac{1}{1}$}}}
\put(-125,80){\makebox(0,0){{\small $\frac{0}{1}$}}}
\put(125,80){\makebox(0,0){{\small $\frac{1}{1}$}}}
\put(-125,60){\makebox(0,0){{\small $\frac{0}{1}$}}}
\put(125,60){\makebox(0,0){{\small $\frac{1}{1}$}}}
\put(-125,60){\makebox(0,0){{\small $\frac{0}{1}$}}}

\put(-120,33){\makebox(0,0){{\small $\frac{0}{1}$}}}
\put(120,33){\makebox(0,0){{\small $\frac{1}{1}$}}}

\put(-40,87){\makebox(0,0){{\small $\frac{1}{3}$}}}
\put(40,87){\makebox(0,0){{\small $\frac{1}{2}$}}}
\put(-45,60){\makebox(0,0){{\small $\frac{1}{3}$}}}
\put(45,60){\makebox(0,0){{\small $\frac{1}{2}$}}}
\put(-40,33){\makebox(0,0){{\small $\frac{1}{3}$}}}
\put(40,33){\makebox(0,0){{\small $\frac{1}{2}$}}}
\put(-93.7,67){\makebox(0,0){{\small $\frac{1}{5}$}}}
\put(-66.6,67){\makebox(0,0){{\small $\frac{1}{4}$}}}
\put(-13.4,67){\makebox(0,0){{\small $\frac{2}{5}$}}}
\put(13.4,67){\makebox(0,0){{\small $\frac{3}{7}$}}}
\put(66.6,67){\makebox(0,0){{\small $\frac{3}{5}$}}}
\put(93.7,67){\makebox(0,0){{\small $\frac{2}{3}$}}}

\put(-111.14,33){\makebox(0,0){{\small $\frac{1}{7}$}}}
\put(-102.56,33){\makebox(0,0){{\small $\frac{1}{6}$}}}
\put(-93.7,33){\makebox(0,0){{\small $\frac{1}{5}$}}}
\put(-75.46,33){\makebox(0,0){{\small $\frac{3}{13}$}}}
\put(-84.84,33){\makebox(0,0){{\small $\frac{2}{9}$}}}
\put(-66.6,33){\makebox(0,0){{\small $\frac{1}{4}$}}}
\put(-57.74,33){\makebox(0,0){{\small $\frac{3}{11}$}}}
\put(-48.86,33){\makebox(0,0){{\small $\frac{2}{7}$}}}
\put(-40,33){\makebox(0,0){{\small $\frac{1}{3}$}}}
\put(-31.14,33){\makebox(0,0){{\small $\frac{3}{8}$}}}
\put(-22.26,33){\makebox(0,0){{\small $\frac{5}{13}$}}}
\put(-13.4,33){\makebox(0,0){{\small $\frac{2}{5}$}}}
\put(-4.54,33){\makebox(0,0){{\small $\frac{7}{17}$}}}
\put(4.54,33){\makebox(0,0){{\small $\frac{5}{12}$}}}
\put(13.4,33){\makebox(0,0){{\small $\frac{3}{7}$}}}
\put(31.14,33){\makebox(0,0){{\small $\frac{5}{11}$}}}
\put(22.26,33){\makebox(0,0){{\small $\frac{4}{9}$}}}
\put(48.86,33){\makebox(0,0){{\small $\frac{5}{9}$}}}
\put(57.74,33){\makebox(0,0){{\small $\frac{4}{7}$}}}
\put(66.6,33){\makebox(0,0){{\small $\frac{3}{5}$}}}
\put(84.84,33){\makebox(0,0){{\small $\frac{7}{11}$}}}
\put(75.46,33){\makebox(0,0){{\small $\frac{5}{8}$}}}
\put(93.7,33){\makebox(0,0){{\small $\frac{2}{3}$}}}
\put(102.56,33){\makebox(0,0){{\small $\frac{5}{7}$}}}
\put(111.14,33){\makebox(0,0){{\small $\frac{3}{4}$}}}

\put(-140,100){\makebox(0,0){{\small $\YY_0$}}}
\put(-140,80){\makebox(0,0){{\small $\YY_1$}}}
\put(-140,60){\makebox(0,0){{\small $\YY_2$}}}
\put(-140,40){\makebox(0,0){{\small $\YY_3$}}}

\end{picture}
\end{center}
\caption{The ECF array $\DD_E$} \label{Figure4}
\end{figure}
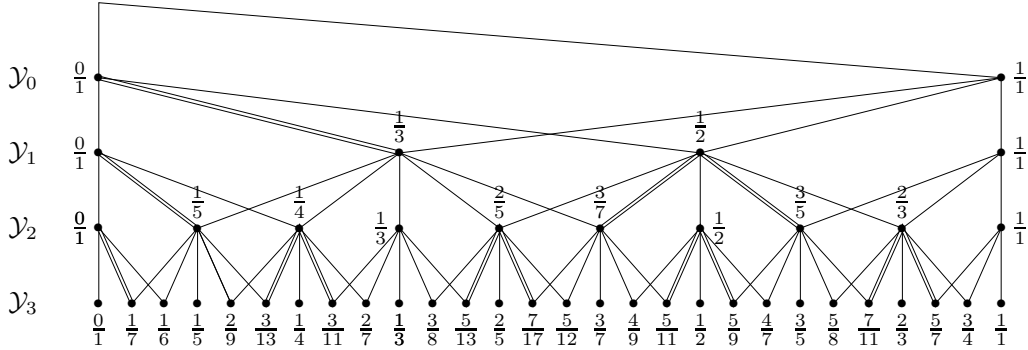

Note that if $x, y \in \ZZ_k$, and $x<y$, then $[(x), (e,a)] < [(y), (\epsilon, \alpha)]$ for $e, \epsilon \in \{-1,1\}$ and $a, \alpha \in \{1\}\cup 2\Z$.
Inductively, this holds for any two continued fractions with initial expansions equal to those of $x$ and $y$, respectively.

Hence we may obtain the ordered set $\YY_{k+1}$ from $\YY_k$ by replacing $0$ with $0$, $[(1, c_k(0)), (1,1)]$, $[(1, c_k(0))]$, and each of the nonzero elements
$x = [(e_1,a_1), (e_2, a_2), \ldots ,(e_n, a_n)] \in \ZZ_k$ with the following five elements:
\begin{equation*}
[(x), (1, c_k(x))],\  [(x), (1, c_k(x)), (1,1)],\  x,\  [(x), (-1, c_k(x)), (1,1)], \  [(x), (-1, c_k(x))],
\end{equation*}
which are in this order if $(-e_1)\cdots (-e_n) = -1$ and are in the reverse order if $(-e_1)\cdots (-e_n) = 1$.
By induction, we have that for any $x \in \ZZ_k$, the neighbors of $x$ in $\YY_k$ are $[(x), (1, c_{k-1}(x)), (1,1)]$ and $[(x), (-1, c_{k-1}(x)), (1,1)]$,
with the understanding that if $x \in \XX_k$ and $c_{k-1}(x) =0$, we have $[(x), (1, 0), (1,1)]=[(x), (1,1)]$ and $[(x), (-1, 0), (1,1)]= [(x), (-1,1)]$.
Combining the fundamental recurrence relations for convergents (see \cite{Kra} equation 1.8) with the definition of the mediant, we quickly obtain the identities  $$[(x), (\epsilon, c_k(x))] = x \oplus [(x), (\epsilon, c_{k-1}(x)), (1,1)]$$ and $$[(x), (\epsilon, c_k(x)), (1,1)]= x \oplus [(x), (\epsilon, c_k(x))],$$ where $\epsilon \in \{-1,1\}$ and $\oplus$ denotes the mediant.

To summarize, we can construct $\YY_{k+1}$ from $\YY_k$ by inserting between each pair of elements (say, $\frac{p}{q} \in \ZZ_k$ and $\frac{r}{s} \in \YY_k \setminus \ZZ_k$)
the successive mediants $\frac{p}{q} \oplus \frac{r}{s} = \frac{p+r}{q+s}$ and $\frac{p}{q} \oplus \frac{p+r}{q+s} = \frac{2p+r}{2q+s}$.
This ECF analogue $\DD_E$ of the classical Stern-Brocot array (also called the  Pascal triangle with memory), is illustrated in
Figure \ref{Figure4}. At every level $n$, the interval $[0,1]$ is partitioned into $3^n$ subintervals.
The appearance of $\frac{2p+r}{2q+s}$ is indicated by a double edge.

\subsection{The even Minkowski type question mark function $Q_E$}
We are now ready to define the ECF analogue of Minkowski's question mark function,
and prove an explicit formula for it in terms of the ECF expansion.
\begin{definition}
For $x \in \mathcal{Y}_k$, define $$Q_E(x) := \frac{|\{y \in \mathcal{Y}_k : y < x\}|}{3^k}.$$
\end{definition}

\begin{proposition}
$Q_E(x)$ does not depend on the choice of $k$, hence $Q_E$ is well-defined on $\mathbb{Q} \cap [0,1]$.
\end{proposition}

\begin{proof} \emph{Case 1.} Suppose $x \in \mathcal{Z}_k$. Then
\begin{equation*}
|\{y \in \mathcal{Y}_k : y < x\}| = 2|\{z \in \mathcal{Z}_k : z <x\}| \quad
\mbox{\rm and} \quad
|\{z \in \mathcal{Z}_{k+1} : z <x\}| = 3|\{z \in \mathcal{Z}_k : z <x\}|.
\end{equation*}
The last formula follows from the characterization of $\mathcal{Z}_k$ in equation \eqref{eq3}.
Indeed, if $x, z \in \mathcal{Z}_k$ and $0 <z <x$, then $[(z), (\pm 1, c_k(z))] <x$, and exactly one of
$[(x), (\pm 1, c_k(x))]$ is less than $x$. We therefore have $|\{y \in \mathcal{Y}_{k+1} : y <x\}| = 3|\{y \in \mathcal{Y}_{k} : y <x\}|$,
so by induction, $3^{-k} |\{y \in \mathcal{Y}_k : y < x\}| = 3^{-k-j} |\{y \in \mathcal{Y}_{k+j} : y < x\}|$
for any $j \in \mathbb{N}$, and so $Q_E(x)$ is well-defined.

\emph{Case 2.} Suppose $x \notin \mathcal{Z}_k$. Then
\begin{equation*}
|\{y \in \mathcal{Y}_k : y < x\}| = 2|\{z \in \mathcal{Z}_k : z <x\}| -1,  \quad
|\{z \in \mathcal{Z}_{k+1} : z <x\}| = 3|\{z \in \mathcal{Z}_k : z <x\}|-1,
\end{equation*}
and so
\begin{equation*}
\begin{split}
|\{y \in \mathcal{Y}_{k+1} : y <x\}| & = 2|\{z \in \mathcal{Z}_{k+1} : z <x\}|-1 \\
& = 6|\{z \in \mathcal{Z}_k : z <x\}|-3 = 3|\{y \in \mathcal{Y}_{k} : y <x\}|.
\end{split}
\end{equation*}
As in the previous case, we conclude by induction that $Q_E(x)$ is well-defined.
\end{proof}

Remark that $Q_E (\frac{1}{2k})=\frac{2}{3^k}$, $Q_E (\frac{1}{2k+1})=\frac{1}{3^k}$,
$Q_E (\frac{5}{13})=\frac{2}{3}-\frac{2}{3^2} -\frac{2}{3^3}+\frac{1}{3^3}=
\frac{1}{3}+\frac{1}{3^2}-\frac{1}{3^3}=\frac{11}{27}$, and that (see also Figure \ref{Figure4})
\begin{equation*}
Q_E (\YY_k) = \Big\{ \frac{m}{3^k} : m=0,1,\ldots,3^k \Big\} .
\end{equation*}

\begin{theorem}\label{Thm1}
Let $x = [(e_1,a_1), (e_2, a_2), \ldots ,(e_n, a_n)]$. Then
\begin{equation*}
Q_E(x) = -\sum_{k=1}^{n} \frac{w_k (-e_1)\cdots (-e_k)}{3^{\sum_{i=1}^{k} \lfloor a_i /2 \rfloor}},
\end{equation*}
where $w_k =2$ if $a_k \in 2\mathbb{N}$ and $w_k=1$ if $a_k =1$.
\end{theorem}

\begin{proof}
Let $y = [(x), (e_{n+1}, a_{n+1})]$, and let $m = \sum_{i=1}^{k} \lfloor a_i /2 \rfloor$ so that $x \in \mathcal{Z}_m$.

\emph{Case 1.} Suppose $a_{n+1} =1$. In this case, $e_{n+1}=1$ as well, so $y \in \mathcal{X}_m$. In the ordered $\mathcal{Y}_m$, $y$
is adjacent to $x$. If $(-e_1)\cdots (-e_n) =1$ then $y >x$, and if $(-e_1)\cdots (-e_n) =-1$ then $y <x$. Hence
\begin{equation*}
Q_E(y) = Q_E(x) - \frac{(-e_1)\cdots (-e_{n+1}) }{3^m}.
\end{equation*}

\emph{Case 2.} Suppose $a_{n+1} = 2j$. In this case, $y \in \mathcal{X}_{m+j}$. At this level, the neighbors of $x$ are
$[(x), (1, 2j)]$, $[(x), (1, 2j), (1,1)]$, $x$, $[(x), (-1, 2j), (1,1)]$, $[(x), (-1, 2j)]$ in this order if
$(-e_1)\cdots (-e_n) = -1$, and in the opposite order if $(-e_1)\cdots (-e_n) = 1$. Hence
\begin{equation*}
Q_E(y) = Q_E(x) - \frac{2(-e_1)\cdots (-e_{n+1}) }{3^{m+j}}.
\end{equation*}
Working backwards from the tail of the continued fraction, repeated application of these relations yields the formula stated above.
\end{proof}

We will see that, by continuity, the formula also holds for infinite even continued fraction expansions, with the finite sum replaced by an infinite one.
Since $Q_E(\Q \cap [0,1])$ is dense in $[0,1]$, the continuity of $Q_E$ proved below will also imply that $Q_E$ is strictly increasing on $[0,1]$.
For rationals which have both an infinite and a finite even continued fraction expansion, the infinite expansion is obtained from the finite one by
replacing the last term $[ \ldots (1,1)]$ with $ [ \ldots (1,2), (-1,2),(-1,2), \ldots ]$. Using the equality $ \sum_{k=1}^{\infty} \frac{2}{3^k} =1$,
it is straightforward to check that the two sums coincide.

\begin{theorem}\label{evencont}
$Q_E(x)$ is H\"older continuous, with best exponent $\frac{\log 3}{2 \log (1+ \sqrt{2})}$.
\end{theorem}

Before proving this, we need a fact about the growth of the ECF continuants.

\begin{proposition}\label{evendenom}
Let $\frac{p_n}{q_n} = [(1,a_1), (e_2, a_2), \ldots , (e_n, a_n)]$ and let $\theta = 1+ \sqrt{2}$. Then
$$q_n < \theta^{(a_1+\cdots+a_n)/2}.$$
\end{proposition}

\begin{proof}
Observe that $q_1 = a_1 < \theta^{a_1/2}$ holds for all $a_1 \in \mathbb{N}$, and $q_0 =1 = {\theta}^0$.
We have the relation $q_k = a_kq_{k-1} + e_{k}q_{k-2}$. Assuming the claim holds for $n =k-2, k-1$, then
\begin{equation*}
a_kq_{k-1} + e_{k}q_{k-2} \leq a_k \theta^{(a_1+\cdots+a_{k-1})/2} + \theta^{(a_1+\cdots+a_{k-2})/2} .
\end{equation*}
So it is sufficient to show that
\begin{equation*}
a_k \theta^{a_{k-1}/2} +1 \leq \theta^{(a_k + a_{k-1})/2},
\end{equation*}
or equivalently that
\begin{equation*}
a_k + \theta^{-a_{k-1}/2} \leq \theta^{a_k/2}.
\end{equation*}
Since $a_k \neq 0$, we must have  $a_{k-1} \geq 2$, so it is sufficient to prove
\begin{equation*}
a_k + \frac{1}{\theta} \leq \theta^{a_k/2},
\end{equation*}
which is always true: we verify that
\begin{equation*}
1 + \frac{1}{\theta} < \theta^{1/2} \text { and } 2 + \frac{1}{\theta} = \theta.
\end{equation*}
For $a_i >2$, it is sufficient to observe that $\theta^{x/2} -x$ is increasing for $x \geq 2$, with derivative
$\frac{1}{2} \theta^{x/2}\log(\theta) -1>0$.
\end{proof}

\begin{rmk*} The exponent in the proposition is the best possible, and it is attained by the convergents of $\sqrt{2}-1 = [(1,2),(1,2),(1,2),\ldots ]$.
\end{rmk*}

\begin{proof} Notice that since each $a_i =2$, the denominators satisfy the recurrence relation $q_k = 2q_{k-1} + q_{k-2}$, and hence are given
by the sequence $1,1,3,7,17,\ldots$, which has the closed form
\begin{equation*}
q_k = \frac{(1 + \sqrt{2})^k + (1 - \sqrt{2})^k}{2}.
\end{equation*}
Asymptotically, $q_k \sim \frac{1}{2}\theta^k$, so the bound $q_k \leq \theta^{(a_1+\cdots+a_k)/2} = \theta^k$ cannot be improved.
\end{proof}

\begin{proof}[Proof of Theorem \ref{evencont}]
Let $x < x'$ in $\Q \cap [0,1]$, and let $y = Q_E(x)$, $y'= Q_E(x')$. Consider $\mathcal{Y}_k$ for the first $k$ such that we have
$x \leq r  < r' \leq x'$ for some $r, r' \in \mathcal{Y}_k$. From the bound on the denominators proved in Proposition $\ref{evendenom}$
we must have $x' -x \geq r' -r \geq \frac{1}{\theta^{2k+2}}$ since $r'$ and $r$ are distinct rationals, each with denominator
at most $\theta^{k+1}$. Since there can be at most 5 elements of $\mathcal{Y}_k$ between
$x$ and $x'$, we have $y' -y \leq \frac{6}{3^k}$. These yield
$$y' -y < (e^{1+ \log_3 6})(x'-x)^{\frac{\log 3}{2 \log \theta}}.$$
To see that this is the best possible exponent, consider $x = \sqrt{2}-1 = [(1,2),(1,2),(1,2), \ldots]$. Let $\frac{p_k}{q_k}$ be the $k$th convergent of $x$.
We have $Q_E(x) = \sum_{j=1}^{\infty} \frac{2 (-1)^{j+1}}{3^j}=\frac{1}{2}$
and $Q_E(\frac{p_k}{q_k}) = \sum_{j=1}^{k} \frac{2 (-1)^{j+1}}{3^j}$, so $|Q_E(x) - Q_E(\frac{p_k}{q_k})|$ is of order
$\frac{1}{3^k}$. Using $q_{k+1}>q_k$, observe that $|x - \frac{p_k}{q_k}| < |\frac{p_{k+1}}{q_{k+1}} - \frac{p_k}{q_k}|< \frac{1}{q_k^2}$.
We know that $q_k$ is of the same order as $\theta^k$, so we have

\begin{equation*}
\bigg| x - \frac{p_k}{q_k}\bigg|^{\frac{\log 3}{2 \log \theta}} \lesssim \theta^{-2k \cdot \frac{\log 3}{2 \log \theta}} = \frac{1}{3^k}.
\end{equation*}
Hence the exponent $\frac{\log 3}{2 \log \theta} = \frac{\log 3}{2 \log (1 +\sqrt{2})}$ is best possible.
\end{proof}
%\qed

\begin{theorem}
$Q_E(x)$ is singular.
\end{theorem}

\begin{proof}
Let $x = [(e_1,a_1), (e_2, a_2), \ldots ]$ with ECF convergents
$\frac{p_n}{q_n} = [(e_1,a_1), \ldots ,(e_n, a_n)]$,
and let $Q_E(x) =y$. Let also $t_n := [(e_{n+2}, a_{n+2}), (e_{n+3}, a_{n+3}), \ldots ]$. We have
\begin{equation*}
\begin{split}
x & = \frac{(a_{n+1} +t_n)p_n + e_{n+1} p_{n-1}}{(a_{n+1} +t_n) q_n + e_{n+1} q_{n-1}}
\quad \mbox{\rm and (see \cite{Kra})} \\
\bigg| x - \frac{p_n}{q_n} \bigg| & = \bigg|\frac{e_{n+1} (q_n p_{n-1} - p_n q_{n-1})}{q_n((a_{n+1} +t_n) q_n + e_{n+1} q_{n-1})}\bigg|
= \frac{1}{q_n((a_{n+1} +t_n) q_n + e_{n+1} q_{n-1})}.
\end{split}
\end{equation*}
Since $|t_n| \leq 1$, in the case where $a_{n+1} > 2$ we have the inequalities
\begin{equation*}
\frac{1}{q_n^2(a_{n+1}+2)} < \bigg| x - \frac{p_n}{q_n} \bigg| < \frac{1}{q_n^2(a_{n+1}-2)}.
\end{equation*}
In the case where $a_{n+1} =2$, we still have
\begin{equation*}
\bigg| x - \frac{p_n}{q_n} \bigg| \leq\frac{1}{q_n^2(1 - \frac{q_{n-1}}{q_n})} \leq \frac{1}{q_n^2(\frac{a_n}{a_n +1})} < \frac{3}{2q_n^2} .
\end{equation*}
Applying the formula for $Q_E(x)$, we have $y- Q_E(\frac{p_n}{q_n}) =
-2 \sum_{k=n+1}^{\infty} \frac{(-e_1)\cdots (-e_k) }{3^{(a_1+\cdots +a_k)/2 }}$ and so
\begin{equation*}
\frac{1}{3^{(a_1+\cdots+a_{n+1})/2}} < \bigg| y - Q_E \bigg( \frac{p_n}{q_n}\bigg) \bigg| \leq \frac{1}{3^{-1+(a_1+\cdots+a_{n+1})/2}}.
\end{equation*}

Letting $r_n = \big| \frac{y -Q_E(p_n/q_n)}{x-p_n/q_n } \big|$ we have
\begin{equation*}
\begin{split}
\frac{ q_n^2 (a_{n+1}-2)}{3^{(a_1+\cdots+a_{n+1})/2}} & < r_n < \frac{ q_n^2 (a_{n+1}+2)}{3^{-1+(a_1+\cdots+a_{n+1})/2}},
\quad \mbox{\rm and} \\
\frac{r_{n}}{r_{n-1}} & <\frac{ q_n^2 (a_{n+1}+2)}{3^{-1+(a_1+\cdots+a_{n+1})/2}} \cdot  \frac{2\cdot3^{(a_1+\cdots+a_{n})/2}}{ 3 q_{n-1}^2}
 = \frac{2 (a_{n+1} +2)}{ 3^{a_{n+1}/2}} \cdot \left(\frac{q_{n}}{q_{n-1}}\right)^2 \\
 & < \frac{2 (a_{n+1} +2)}{ 3^{a_{n+1}/2}} \cdot (a_n +1)^2
 < \frac{2 (a_{n+1} +2)(a_n+1)^2}{ 3^{a_{n+1}/2}}.
\end{split}
\end{equation*}

If for some $x$ the $a_i$ are unbounded, then we may consider the subsequence  $a_{i_k}$ where $i_1 = \inf\{i : a_i  >2\}$ and
$i_{k+1} = \inf\{i : a_i  >a_{i_{k}}\}$. Then for every $k$ we have $2 <a_{i_k}$ and $a_{i_k -1} < a_{i_k}$, so the above will imply that
$$\frac{r_{{i_k}-1}}{r_{{i_k}-2}} <\frac{2 (a_{i_k} +2)(a_{i_k-1}+1)^2}{ 3^{a_{i_k}/2}} <\frac{2 (a_{i_k} +2)^3}{ 3^{a_{i_k}/2}},$$
which converges to $0$. This implies that if the derivative of $?(x)$ exists and is finite,
it must be equal to $0$. As we will see in the next proposition, the $a_i$ are in fact unbounded for almost every $x$.
Since $Q_E(x)$ is monotone, the derivative must in fact exist almost everywhere, and hence $Q_E(x)$ is singular.
\end{proof}

\begin{proposition}
The set of $x$ with bounded even partial quotients has measure $0$.
\end{proposition}
\begin{proof}
It is well-known that almost every number is normal with respect to the regular continued fraction. (The results of \cite{KraNak} can
perhaps be extended to show that this in fact implies being normal with respect to the even continued fraction, although we only need a much weaker result.)

For each $k>0$, every number which is normal with respect to the regular continued fraction expansion will have at some point in its regular continued fraction expansion two consecutive
$a_i, a_{i+1} > k$. When applying the singularization and insertion algorithm (see  \cite{Mas} Section 1.3) to obtain the even continued fraction expansion, partial quotients which are greater than $1$ are either increased, or replaced by a sequence of $(-1,2)$ terms. Since the algorithm cannot replace two consecutive partial quotients in this way, we must end up with at least one even partial quotient $a_j >k$. Hence almost every number has unbounded even partial quotients.
\end{proof}

\subsection{The linearization of the map $F_E$}
The formula proved in Theorem \ref{Thm1} and the continuity of $Q_E$ provide the formula
\begin{equation}\label{eq4}
Q_E ( [(1,2k_1),(e_1,2k_2),(e_2,2k_3),\ldots ]) =
-2\sum\limits_{n=1}^\infty \frac{(-e_1)\cdots (-e_n)}{3^{k_1+\cdots +k_n}} .
\end{equation}

Consider the continuous piecewise linear maps $\overline{F}_E, \overline{T}_E :[0,1]\rightarrow [0,1]$ defined by
\begin{equation*}
\overline{F}_E (y) =\begin{cases} 3y & \mbox{\rm if $y\in [0,\frac{1}{3}]$} \\
2-3y & \mbox{\rm if $y\in [\frac{1}{3},\frac{2}{3}]$} \\
3y-2 & \mbox{\rm if $y\in [\frac{2}{3},1]$}\end{cases} \quad \mbox{\rm and} \quad
\overline{T}_E (y) =\begin{cases} 2-3^k y & \mbox{\rm if $y\in [3^{-k}, 2\cdot 3^{-k}]$} \\
3^k y -2 & \mbox{\rm if $y\in [2\cdot 3^{-k},3^{-k+1}].$}
\end{cases}
\end{equation*}

\begin{proposition}\label{prop5}
The homeomorphism $Q_E$ of $[0,1]$ linearizes the maps $F_E$ and $T_E$ as follows:
\begin{equation*}
\mbox{\rm (i)}\  Q_E F_E Q_E^{-1} = \overline{F}_E,\qquad
\mbox{\rm (ii)}\  Q_E T_E Q_E^{-1} = \overline{T}_E.
\end{equation*}
\end{proposition}

\begin{proof}
Let $x=[(1,2k_1),(e_1,2k_2),(e_2,2k_3),\ldots ] \in (0,1)$ and employ repeatedly formula \eqref{eq4}.

(i) There are three cases to be considered:

\emph{Case 1.} $x\in (\frac{1}{2},1)$, where $k_1=1$ and $e_1=-1$. Then we successively infer
\begin{equation*}
\begin{split}
Q_E (x) & =Q_E \big( [(1,2),(-1,2k_2),(e_2,2k_3),\ldots ]\big) =
2\bigg( \frac{1}{3}+\frac{1}{3^{k_2+1}} -\frac{e_2}{3^{k_2+k_3+1}} + \cdots \bigg),\\
(\overline{F}_E Q_E)(x) & = 3Q_E (x)-2 =
2\bigg( \frac{1}{3^{k_2}} -\frac{e_2}{3^{k_2+k_3}} +\frac{e_2e_3}{3^{k_2+k_3+k_4}} -\cdots \bigg) , \\
(Q_E F_E)(x) & = Q_E \big( [(1,2k_2),(e_2,2k_3),(e_3,2k_4),\ldots ]\big) \\
& =2\bigg( \frac{1}{3^{k_2}}-\frac{e_2}{3^{k_2+k_3}} +\frac{e_2e_3}{3^{k_2+k_3+k_4}} -\cdots \bigg)
=(\overline{F}_E Q_E)(x).
\end{split}
\end{equation*}

\emph{Case 2.} $x\in (\frac{1}{3},\frac{1}{2})$, where $k_1=1$, $e_1=1$, and we have
\begin{equation*}
\begin{split}
Q_E (x) & =Q_E \big( [(1,2),(1,2k_2),(e_2,2k_3),\ldots ]\big) =
2\bigg( \frac{1}{3} -\frac{1}{3^{k_2+1}}+\frac{e_2}{3^{k_2+k_3+1}} -\cdots \bigg) , \\
(\overline{F}_E Q_E)(x) & =2-3Q_E(x) =2\bigg( \frac{1}{3^{k_2}} -
\frac{e_2}{3^{k_2+k_3}} +\frac{e_2 e_3}{3^{k_2+k_3+k_4}} -\cdots \bigg)
=(Q_E F_E)(x).
\end{split}
\end{equation*}

\emph{Case 3.} $x\in (0,\frac{1}{3})$, where $k_1 \geq 2$ and we have
\begin{equation*}
\begin{split}
Q_E(x) & =2\bigg( \frac{1}{3^{k_1}} -\frac{e_1}{3^{k_1+k_2}} +\frac{e_1e_2}{3^{k_1+k_2+k_3}} -\cdots \bigg) ,\\
(\overline{F}_E Q_E)(x) & =3 Q_E(x) =2\bigg( \frac{1}{3^{k_1-1}} -\frac{e_1}{3^{k_1-1+k_2}} +
\frac{e_1e_2}{3^{k_1-1+k_2+k_3}} -\cdots \bigg) ,\\
(Q_E F_E)(x) & = Q_E \big( [(1,2k_1-2),(e_1,2k_2),(e_2,2k_3),\ldots ]\big) \\
& =2 \bigg( \frac{1}{3^{k_1-1}} -\frac{e_1}{3^{k_1-1+k_2}} +\frac{e_1e_2}{3^{k_1-1+k_2+k_3}} - \cdots \bigg)
=(\overline{F}_E Q_E)(x).
\end{split}
\end{equation*}

(ii) follows by direct verification along the line of (i), considering the cases
$x\in ( \frac{1}{2k+1},\frac{1}{2k})$ where $k_1=k$, $e_1=1$ and $Q_E (x)\in [3^{-k},2\cdot 3^{-k}]$,
and respectively $x\in (\frac{1}{2k},\frac{1}{2k-1})$ where $k_1=k$, $e_1=-1$ and
$Q_E(x)\in [2\cdot 3^{-k}, 3^{-k+1}]$.

As suggested by one of the referees, (ii) can also be directly deduced from (i) by a dynamic argument,
since $T_E$ and $\overline{T}_E$ are conjugated to the first return map of $F_E$
and respectively $\overline{F}_E$ on $[\frac{1}{3},1]$, and $Q_E$ maps $[\frac{1}{3},1]$ onto $[\frac{1}{3},1]$.
\end{proof}

\subsection{The ECF Stern Sequence and Stern Polynomials}
We now consider the integer sequence of denominators of the fractions in our analogue $\DD_E$ of the Stern-Brocot array, giving an ECF version of the Stern sequence (A002487 in \cite{OEIS}).
As we will see, this ends up being closely related to a triadic version of the Stern sequence that has been constructed by Northshield in \cite{Nor}.
It is convenient to work on $[-1,1)$, since $|\{x = [(e_1,a_1), \ldots, (e_n,a_n)] \in \Q \cap [-1,1) : \sum_{i=1}^n a_i \leq 2k +1 \}| = 2 \cdot 3^k$, so $n \mapsto 3n$ corresponds to moving down a level in the extension of the diagram $\DD_E$ to $[-1,1)$. Let $\{\beta_n\}$ be the sequence of the denominators of the fractions in the extension of $\mathcal{D}_E$ to $[-1,1)$, reading each row from left to right. From the structure of $\DD_E$, we obtain the relations
\begin{equation*}
\beta_{3n} = \beta_n, \qquad
\beta_{3n+1} = w(n)\beta_n + \beta_{n+1}, \qquad
\beta_{3n+2} = \beta_n + w(n+1)\beta_{n+1}, \\
\end{equation*}
where $w(n) =2$ if $n$ is even and $1$ if $n$ is odd. We let $\beta_0 =0$, and observe that our $\{\beta_n\} = 	0, 1, 1, 1, 2, 3, 1, 3, 2, 1, 3, 5, 2, 7,\ldots$
is the sequence A277750 in \cite{OEIS}. From the above relations we derive
\begin{equation*}
\begin{split}
B_o(x) := & \sum_{n \text{ odd }} \beta_n x^n = \sum_{n \text{ odd }} \beta_{3n} x^{3n} +\sum_{n \text{ even }}\beta_{3n+1} x^{3n+1} +\sum_{n \text{ odd }} \beta_{3n+2} x^{3n+2}\\
= & (x^{-2} +1 + x^{2})\sum_{n \text{ odd }} \beta_n x^{3n} + 2(x^{-1} +x)\sum_{n \text{ even }} \beta_n x^{3n} ,\qquad \mbox{\rm and}  \\
B_e(x) := & \sum_{n \text{ even }} \beta_n x^n = \sum_{n \text{ even }} \beta_{3n} x^{3n} +\sum_{n \text{ odd }}\beta_{3n+1} x^{3n+1} +\sum_{n \text{ even }} \beta_{3n+2} x^{3n+2}\\
= & (x^{-2} +1 + x^{2})\sum_{n \text{ even }} \beta_nx^{3n} + (x^{-1} +x)\sum_{n \text{ odd }} \beta_n x^{3n}.
\end{split}
\end{equation*}
Although we do not immediately obtain an infinite product form for the generating function (as in the case of the Stern sequence), we will see that this is possible
for a slight modification of our sequence. Rewriting the above in matrix form, we have
\[ \left( \begin{matrix}
B_o(x) \\
B_e(x) \end{matrix} \right) =\left( \begin{matrix}
x^{-2} +1 + x^{2} & 2(x^{-1} +x)  \\
x^{-1} +x & x^{-2} +1 + x^{2}  \end{matrix} \right)\left( \begin{matrix}
B_o(x^3) \\
B_e(x^3) \end{matrix} \right). \]
The matrix $\left( \begin{matrix}
x^{-2} +1 + x^{2} & 2(x^{-1} +x)  \\
x^{-1} +x & x^{-2} +1 + x^{2}  \end{matrix} \right)$ has an eigenvector $\left( \begin{matrix}
\sqrt{2} \\
1 \end{matrix}\right)$ with eigenvalue $(x^{-2} + \sqrt{2}x^{-1} +1 + \sqrt{2}x + x^2)$,
so we obtain the relation
\begin{equation*}
\sqrt{2}B_o(x) + B_e(x) = (x^{-2} + \sqrt{2}x^{-1} +1 + \sqrt{2}x + x^2)(\sqrt{2}B_o(x^3) + B_e(x^3)),
\end{equation*}
from which we obtain the infinite product representation
\begin{equation*}
\sqrt{2}B_o(x) + B_e(x) = \prod_{n=0}^\infty (x^{-2\cdot 3^n} + \sqrt{2}x^{-3^n} +1 + \sqrt{2}x^{3^n} + x^{2\cdot 3^n}).
\end{equation*}
The ``diagonalized" sequence obtained from $\{\beta_n\}$ by multiplying the odd terms by $\sqrt{2}$ is what Northshield denotes
$\{b_n\}$ in \cite{Nor}, where many properties of the sequence are proved, including an infinite product
representation in Section 4. Our $\{ \beta_n\}$ appear as the denominators of Northshield's $R_n$.

Dilcher and Stolarsky considered a polynomial version of the Stern sequence in \cite{DilSto}.
The ECF Stern sequence can be similarly generalized, by setting $\beta(0,x) =0$, $\beta(1,x) =1$, $\beta(2,x) =1$, and
\begin{align*}
\beta(3n,x) &= \beta(n,x^4),\\
\beta(3n+1,x) &= \begin{cases}
            (1+x)\beta(n,x^4) +x^3 \beta(n+1,x^4) & \mbox{\rm if $n$ is even } \\
            \beta(n,x^4) +x^2 \beta(n+1,x^4) & \mbox{\rm if $n$ is odd, }
        \end{cases}
   \\
\beta(3n+2,x) &= \begin{cases}
            \beta(n,x^4) +x^2 \beta(n+1,x^4) & \mbox{\rm if $n$ is even } \\
            \beta(n,x^4) +(x^2+x) \beta(n+1,x^4) & \mbox{\rm if $n$ is odd. }
       \end{cases} \\
\end{align*}
The above relations are derived from replacing the mediant construction with the polynomial version used by Dilcher and Stolarsky.
It is immediate from the definition that $\beta(n,1)$ recovers the ECF Stern sequence $\beta_n$, and that $\beta(n,x)$ has coefficients in $\{0,1\}$.
It would be interesting to find a combinatorial interpretation of the ECF Stern sequence or its polynomial generalization.

\section{Odd Partial Quotients}\label{OCF}
\subsection{Odd continued fraction and associated Gauss and Farey maps}
In this section we consider the OCF in $[-1,1]$ given by
\begin{equation}\label{eq5}
[(e_1, a_1), (e_2, a_2), (e_3, a_3), \ldots ] = \cfrac{e_1}{a_1+\cfrac{e_2}{a_2+\cfrac{e_3}{\ddots}}},
\end{equation}
where $e_i \in \{\pm1\}$, $a_i \in 2\N -1$, $e_1 = 1$, and $a_i + e_{i+1} > 0$. For uniqueness of representations, we require that
in a finite expansion, if the last $a_j =1$, then $e_j =1$. For example, we have
$\frac{1}{2k-1}=[(1,2k-1)]$, $\frac{1}{2k}=[(1,2k-1),(1,1)]$, $\frac{4}{7}=[(1,1),(1,1),(1,3)]$,
$\frac{7}{12}=[(1,1),(1,1),(1,3),(-1,1),(1,1)]$.

We consider the Farey type map $F_O:[0,1]\rightarrow [0,1]$ associated to OCF expansions, given by
\begin{equation}
F_O(x)= \begin{cases} \frac{x}{1-2x} & \mbox{\rm if $0\leq x < \frac{1}{3}$} \\
3-\frac{1}{x} & \mbox{\rm if $\frac{1}{3} \leq x \leq \frac{1}{2}$} \\
\frac{1}{x}-1 & \mbox{\rm if $\frac{1}{2} \leq x\leq 1.$}
\end{cases}
\end{equation}
Symbolically, $F_O$ acts on the OCF representation \eqref{eq5} by subtracting $2$ from the leading digit $a_1$ of $x$ when
$(a_1,e_2)\neq (3,-1)$ and $(a_1,e_1)\neq (1,1)$ (which correspond to $x$ between $0$ and $\frac{1}{3}$), and by removing $(a_1,e_2)$ when
$(a_1,e_2)\in \{ (3,-1),(1,1)\}$
(which corresponds to $x$ between $\frac{1}{3}$ and $1$), i.e.
\begin{equation*}
F_O ([(1,a_1),(e_2,a_2),(e_3,a_3),\ldots]) =\begin{cases}
[(1,a_1-2),(e_2,a_2),\ldots ] & \mbox{\rm if $(a_1,e_2)\notin \{ (3,-1),(1,1)\}$} \\
[(1,a_2),(e_3,a_3),(e_4,a_4),\ldots ] & \mbox{\rm if $(a_1,e_2)\in \{ (3,-1),(1,1)\}$.}
\end{cases}
\end{equation*}

\begin{center}
\begin{figure}
\includegraphics[scale=0.65,bb = 0 20 300 260]{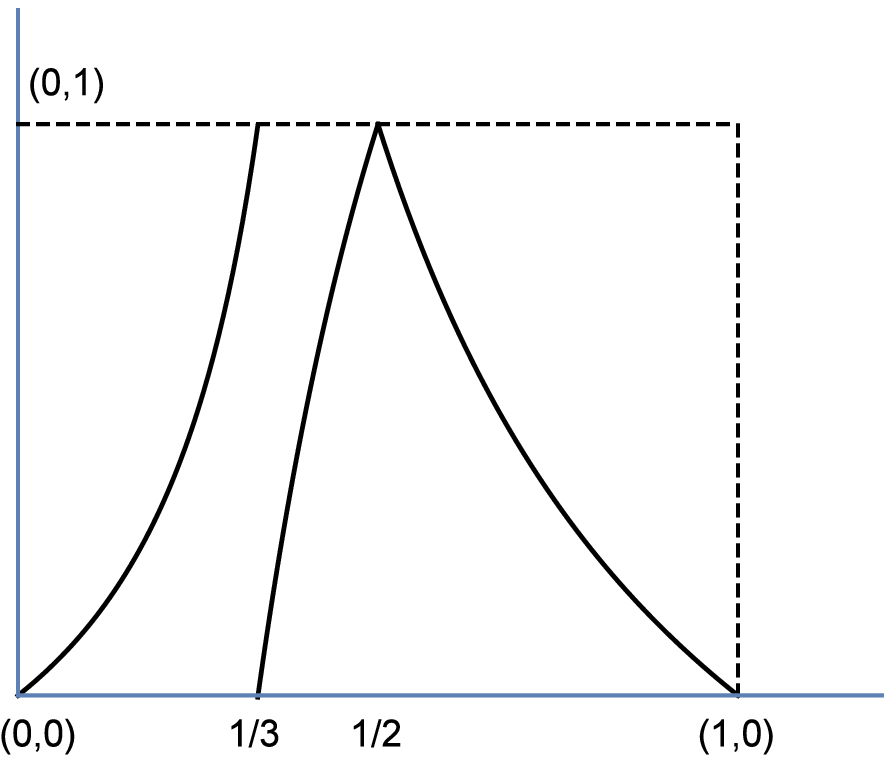}
\includegraphics[scale=0.65,bb = 0 20 200 260]{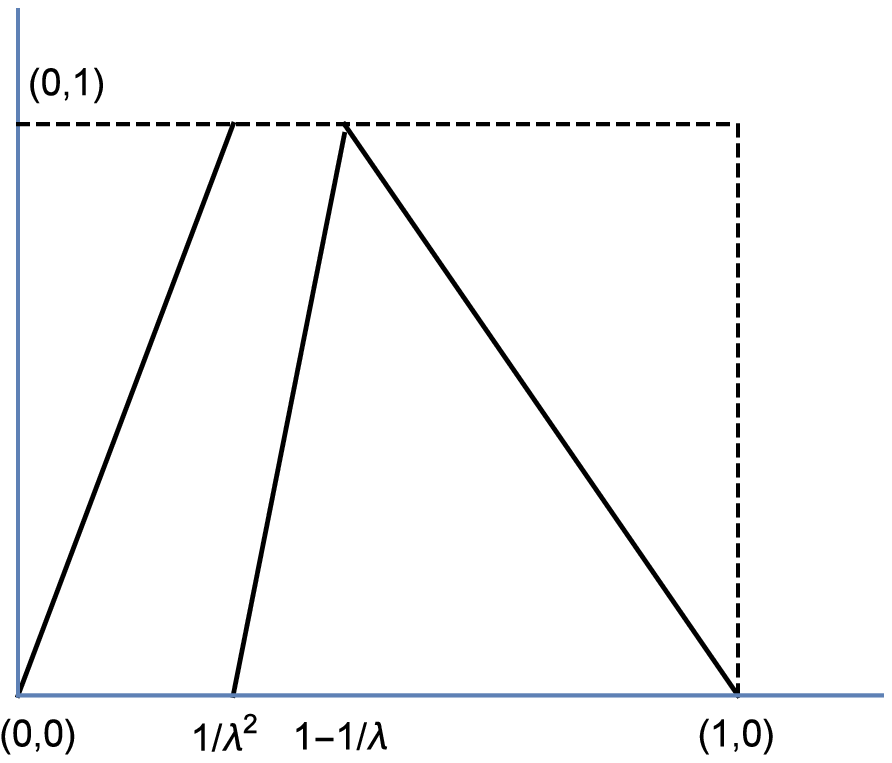}
\caption{The odd Farey map $F_O$ and its linearization $\overline{F}_O$}\label{Figure5}
\end{figure}
\end{center}

The following result follows from direct verification:

\begin{lemma}
The infinite measure $d\nu_O (x)=\frac{1}{x}+\frac{1}{G+1-x}$ is $F_O$-invariant.
\end{lemma}

The first return map $R_O$ of $F_O$ on $[\frac{1}{3},1)$ acts on the OCF expansion as
\begin{equation*}
R_O\big( [(1,a_1),(e_2,a_2),(e_3,a_3),\ldots ]\big) =[(1,a_1),(e_3,a_3),(e_4,a_4),\ldots] ,
\end{equation*}
where $(a_1,e_2)\in \{ (3,-1),(1,1)\}$.
Recall that the OCF Gauss map $T_O$ acts on $[0,1]$ by $T_O(0)=0$ and
\begin{equation*}
T_O (x)= \bigg| \frac{1}{x}-2\bigg[ \frac{1}{2x}\bigg] -1 \bigg| =
\begin{cases}
2k+1- \frac{1}{x} & \mbox{\rm if $x\in [\frac{1}{2k+1},\frac{1}{2k}]$} \\
\frac{1}{x} -(2k-1) & \mbox{\rm if $x\in [\frac{1}{2k},\frac{1}{2k-1}]$,}
\end{cases}
\end{equation*}
and it acts on OCF expansions \eqref{eq5} restricted to $(0,1)$ by
\begin{equation*}
T_O ( [(1,a_1),(e_2,a_2),(e_3,a_3),\ldots ]) =
[(1,a_2),(e_3,a_3),(e_4,a_4),\ldots ].
\end{equation*}
Recall also that $d\mu_O (x) =( \frac{1}{G-1+x}+\frac{1}{G+1-x})dx$ is a finite $T_O$-invariant measure \cite{Sch1}.

We will instead consider the extended OCF Gauss map $\wT_O :[-1,1) \rightarrow [-1,1)$ acting on
the OCF expansion \eqref{eq5} as
\begin{equation*}
\wT_O ( [(e_1,a_1),(e_2,a_2),(e_3,a_3),\ldots ]) =[(e_2,a_2),(e_3,a_3),(e_4,a_4),\ldots ],
\end{equation*}
or equivalently we can take $\wT_O (0)=0$ and
\begin{equation*}
\wT_O (x)= \frac{1}{\lvert x\rvert} - 2 \bigg[ \frac{1}{2\lvert x\rvert} \bigg] -1
\quad \mbox{\rm if $x\neq 0$.}
\end{equation*}
It is plain that $R_O$ is conjugated with $\wT_O$, and more precisely
$\wT_O=\psi R_O \psi^{-1}$, where
$\psi:[\frac{1}{3},1) \rightarrow [-1,1)$ is the invertible map given by
\begin{equation*}
\psi (x)=\begin{cases}
\frac{1}{x}- 3 & \mbox{\rm if $x\in [\frac{1}{3},\frac{1}{2}]$} \\
\frac{1}{x}-1 & \mbox{\rm if $x\in (\frac{1}{2},1]$}
\end{cases} \quad \mbox{\rm with} \quad
\psi^{-1} (y) =\begin{cases} \frac{1}{3+y} & \mbox{\rm if $y\in [-1,0]$} \\
\frac{1}{1+y} & \mbox{\rm if $y\in (0,1)$.}
\end{cases}
\end{equation*}
The push-forward measure $\widetilde{\mu}_O$ of $\nu_O\vert_{[1/3,1)}$ by $\psi$
is $\wT_O$-invariant, being given by
\begin{equation*}
\begin{split}
\int_{1/3}^1 f\big( \psi (x)\big) \, \frac{dx}{x(G+1-x)} & =\int_{1/3}^{1/2} f\bigg( \frac{1}{x}-3\bigg) \frac{dx}{x(G+1-x)}
+\int_{1/2}^1 f\bigg( \frac{1}{x}-1\bigg) \frac{dx}{x(G+1-x)} \\
& = \int_{-1}^1 f(y) \, d\widetilde{\mu}_O (y),
\end{split}
\end{equation*}
that is
\begin{equation*}
d \widetilde{\mu}_O (y) = \frac{1}{G^2} \cdot \frac{dy}{y+G+1}\, \chi_{[-1,0]} +\frac{1}{G^2} \cdot \frac{dy}{y+G-1}\, \chi_{(0,1)} .
\end{equation*}
Again, $\wT_O$ is an extension of $T_O$ with
$\pi \wT_O=T_O \pi$, where
$\pi (x)=\lvert x\rvert$. The push-forward of $\widetilde{\mu}_O$ under $\pi$ is
the $T_O$-invariant measure $\mu_O$.
The map $\wT_O$ coincides with the map $T$ introduced and investigated by
Rieger in Chapters 2 and 3 of \cite{Rie}. Note also that $\rho = G \widetilde{\mu}_O$
is the $T$-invariant measure considered in \cite[Theorem 6.1]{Rie}.

\subsection{The odd Minkowski type question mark function $Q_O$}

Let $\lambda >1$ be the unique real root of $x^3 -x^2 -x -1=0$. Following \cite{Zha} and \eqref{eq5},
we consider the map $Q_O$ on $[0,1]$ by
\begin{equation}\label{eq6}
Q_O([(e_1,a_1),(e_2,a_2),(e_3,a_3),\ldots]) =  -\sum_{k=1}^{\infty} \frac{(-e_1)\cdots (-e_k) }{\lambda^{a_1+\cdots+a_k -1}},
\end{equation}
which coincides with Zhabitskaya's $F^0(x)$. Note that in the rational case we have a finite expansion
$[(e_1,a_1),(e_2,a_2),(e_3,a_3),\ldots ,(e_n,a_n)]$, and the above formula holds for the finite sum.
For example, we have $Q_O(\frac{1}{2k-1})=\lambda^{-2k+2}$, $Q_O(\frac{1}{2k})=\lambda^{-2k+2}-\lambda^{-2k+1}$,
$Q_O(\frac{4}{7})=1-\lambda^{-1}+\lambda^{-4}$,
$Q_O (\frac{7}{12}) = 1-\lambda^{-1}+\lambda^{-4}+\lambda^{-5}-\lambda^{-6}$.

\begin{theorem}\label{oddcont}
$Q_O$ is is H\"older continuous, with best exponent $\frac{\log \lambda}{2 \log G}$.
\end{theorem}

In preparation for this result, we need two preliminary facts about the (ordered) set
\begin{equation}\label{eq7}
\mathcal{Y}_n := \big\{ x \in \mathbb{Q} \cap [0,1] : x = [(e_1, a_1), (e_2, a_2), \ldots ,(e_k, a_k)] \text{ and } a_1+\cdots +a_k \leq n +1 \big\}.
\end{equation}
In this section we use the same notation $\mathcal{Y}_n$ and $\mathcal{X}_n$ as in \cite{Zha}. Note that
the analogue of formula \eqref{eq2} does not hold for $\YY_n$ here because $F_O^{-1}(\{ 0,1\}) \setminus \YY_1 =\{ \frac{1}{3}\}$.

What we need will follow from the structure of the analogue of the Stern-Brocot tree for odd continued fractions, which we denote $\mathcal{D}$, as in \cite{Zha}.

\begin{proposition}\label{eq8}
For a reduced fraction $\frac{p}{q} \in \mathcal{Y}_n$,
$$q \leq G^{n+2}.$$
\end{proposition}

\begin{proof}
In fact, the largest denominator in $\mathcal{Y}_n$ is given by the $(n+2)$-th Fibonacci number. This can be directly verified for the first few $n$, and follows inductively from the fact that every element of $\mathcal{X}_{n+1} :=\mathcal{Y}_{n+1} \setminus \mathcal{Y}_n$ is the mediant of two adjacent elements of $\mathcal{Y}_n$. Since no two elements of $ \mathcal{X}_{n+1}$ are adjacent in $\mathcal{Y}_{n+1}$ (see \cite{Zha} page 9), the largest denominator in $\mathcal{Y}_{n+2}$ is at most the sum of the largest denominator in $\mathcal{Y}_{n+1}$ and the largest denominator in $\mathcal{Y}_{n}$. Since this recurrence relation is in fact satisfied by the
convergents of $G-1 =[(1,1), (1,1), (1,1),\ldots ]$, we obtain the stated (sharp) upper bound for the denominators.
\end{proof}

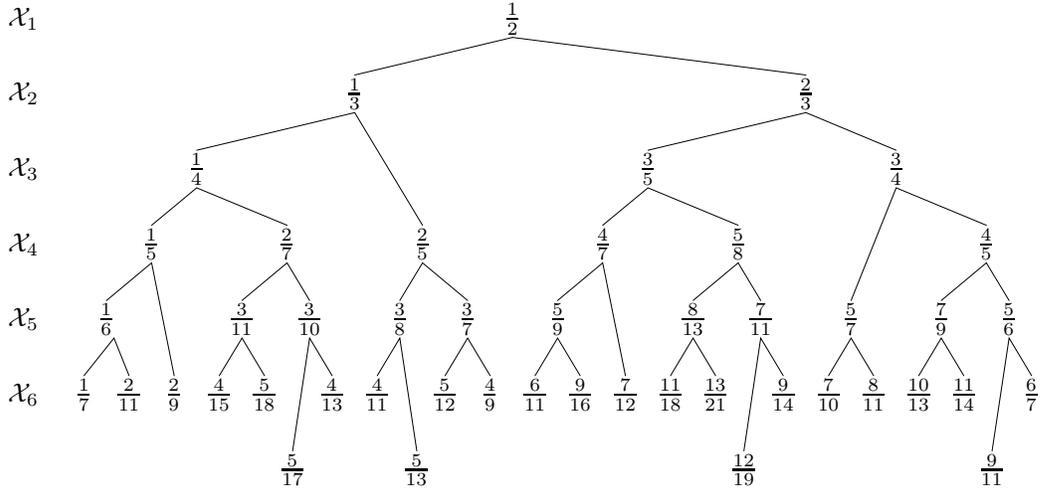
\begin{figure}[htb]
\begin{center}
\unitlength 0.5mm
\begin{picture}(30,120)(0,0)
%\thinlines
%\dottedline{2}(-120,100)(-120,120)(120,100)

\put(0,120){\makebox(0,0){{\small $\frac{1}{2}$}}}
\put(-42,100){\makebox(0,0){{\small $\frac{1}{3}$}}}
\put(78,100){\makebox(0,0){{\small $\frac{2}{3}$}}}

\path(-42,105)(0,115) \path(0,115)(78,105)

\put(-6,20){\makebox(0,0){{\small $\frac{4}{9}$}}}
\put(-12,40){\makebox(0,0){{\small $\frac{3}{7}$}}}
\put(-18,20){\makebox(0,0){{\small $\frac{5}{12}$}}}
\put(-24,60){\makebox(0,0){{\small $\frac{2}{5}$}}}
\put(-30,40){\makebox(0,0){{\small $\frac{3}{8}$}}}
\put(-25.5,0){\makebox(0,0){{\small $\frac{5}{13}$}}}
\put(-36,20){\makebox(0,0){{\small $\frac{4}{11}$}}}

\path(-30,45)(-24,55)(-12,45)
\path(-18,25)(-12,35)(-6,25)
\path(-36,25)(-30,35)(-25.5,5)
\path(-42,95)(-24,65)

\put(-48,20){\makebox(0,0){{\small $\frac{4}{13}$}}}
\put(-54,40){\makebox(0,0){{\small $\frac{3}{10}$}}}
\put(-58.5,0){\makebox(0,0){{\small $\frac{5}{17}$}}}
\put(-60,60){\makebox(0,0){{\small $\frac{2}{7}$}}}
\put(-66,20){\makebox(0,0){{\small $\frac{5}{18}$}}}
\put(-72,40){\makebox(0,0){{\small $\frac{3}{11}$}}}
\put(-78,20){\makebox(0,0){{\small $\frac{4}{15}$}}}
\put(-84,80){\makebox(0,0){{\small $\frac{1}{4}$}}}
\put(-90,20){\makebox(0,0){{\small $\frac{2}{9}$}}}
\put(-96,60){\makebox(0,0){{\small $\frac{1}{5}$}}}
\put(-102,20){\makebox(0,0){{\small $\frac{2}{11}$}}}
\put(-108,40){\makebox(0,0){{\small $\frac{1}{6}$}}}
\put(-114,20){\makebox(0,0){{\small $\frac{1}{7}$}}}

\path(-114,25)(-106,35)(-102,25)
\path(-108,45)(-96,55)(-90,25)
\path(-96,65)(-84,75)(-60,65)
\path(-72,45)(-60,55)(-54,45)
\path(-78,25)(-72,35)(-66,25)
\path(-48,25)(-54,35)(-58.5,5)
\path(-84,85)(-42,95)

\put(6,20){\makebox(0,0){{\small $\frac{6}{11}$}}}
\put(18,20){\makebox(0,0){{\small $\frac{9}{16}$}}}
\put(12,40){\makebox(0,0){{\small $\frac{5}{9}$}}}
\path(6,25)(12,35)(18,25)

\put(24,60){\makebox(0,0){{\small $\frac{4}{7}$}}}
\put(30,20){\makebox(0,0){{\small $\frac{7}{12}$}}}
\path(12,45)(24,55)(30,25)

\put(36,80){\makebox(0,0){{\small $\frac{3}{5}$}}}
\put(60,60){\makebox(0,0){{\small $\frac{5}{8}$}}}
\path(24,65)(36,75)(60,65)

\put(48,40){\makebox(0,0){{\small $\frac{8}{13}$}}}
\put(42,20){\makebox(0,0){{\small $\frac{11}{18}$}}}
\put(54,20){\makebox(0,0){{\small $\frac{13}{21}$}}}
\path(42,25)(48,35)(54,25)

\put(66,40){\makebox(0,0){{\small $\frac{7}{11}$}}}
\path(48,45)(60,55)(66,45)

\put(61.5,0){\makebox(0,0){{\small $\frac{12}{19}$}}}
\put(72,20){\makebox(0,0){{\small $\frac{9}{14}$}}}
\path(61.5,5)(66,35)(72,25)

\put(102,80){\makebox(0,0){{\small $\frac{3}{4}$}}}
\path(36,85)(78,95)(102,85)

\put(84,20){\makebox(0,0){{\small $\frac{7}{10}$}}}
\put(90,40){\makebox(0,0){{\small $\frac{5}{7}$}}}
\put(96,20){\makebox(0,0){{\small $\frac{8}{11}$}}}
\path(84,25)(90,35)(96,25)

\put(126,60){\makebox(0,0){{\small $\frac{4}{5}$}}}
\path(90,45)(102,75)(126,65)

\put(114,40){\makebox(0,0){{\small $\frac{7}{9}$}}}
\put(132,40){\makebox(0,0){{\small $\frac{5}{6}$}}}
\path(114,45)(126,55)(132,45)

\put(108,20){\makebox(0,0){{\small $\frac{10}{13}$}}}
\put(120,20){\makebox(0,0){{\small $\frac{11}{14}$}}}
\path(108,25)(114,35)(120,25)

\put(127.5,0){\makebox(0,0){{\small $\frac{9}{11}$}}}
\put(138,20){\makebox(0,0){{\small $\frac{6}{7}$}}}
\path(127.5,5)(132,35)(138,25)

\put(-130,120){\makebox(0,0){{\small $\XX_1$}}}
\put(-130,100){\makebox(0,0){{\small $\XX_2$}}}
\put(-130,80){\makebox(0,0){{\small $\XX_3$}}}
\put(-130,60){\makebox(0,0){{\small $\XX_4$}}}
\put(-130,40){\makebox(0,0){{\small $\XX_5$}}}
\put(-130,20){\makebox(0,0){{\small $\XX_6$}}}

\end{picture}
\end{center}
\caption{Zhabitskaya's odd Farey tree} \label{Figure6}
\end{figure}

\begin{proposition} \label{Prop8}
There exists a universal constant $C$ such that if $x$ and $y$ are adjacent elements of $\mathcal{Y}_n$, then
$$|Q_O(x) -Q_O(y)| \leq C \lambda^{-n}. $$
\end{proposition}

\begin{proof}First, suppose that $y \in \mathcal{X}_n$. We have already noted in the proof of Proposition \ref{eq8} that no two elements of $\mathcal{X}_{n}$ are adjacent in $\mathcal{Y}_n$,
so it must be the case that $y \in \mathcal{X}_n$ is a descendant of $x$, in the sense that it is obtained from $x$ by (perhaps repeatedly) taking mediants.
Suppose $x = [(e_1, a_1), (e_2, a_2), \ldots, (e_j, a_j)]$. There are three possible ``moves'' in the tree $\mathcal{D}$, each corresponding to a possible
relationship between an element $x \in \mathcal{X}_k$ and its descendant in $\mathcal{X}_{k+1}$ or $\mathcal{X}_{k+2}$. The first type of move is appending
$(1,1)$ to the tail of the continued fraction of $x$. The second (possible only when $a_j >1$) is appending $(-1,1), (1,1)$ to the tail, and the third
(possible only when $a_j =1$) is to remove $(e_j, a_j) = (1,1)$ and replace $(e_{j-1}, a_{j-1})$ with $(e_{j-1}, a_{j-1}+2)$. Suppose we call a move
(of any of the three types) a left move if the result is less than the input, and a right move if the result is greater than the input. Not only is $y$
obtained from $x$ by a series of these moves, but since $y$ is adjacent to $x$, it must be obtained either by a right move followed by only left moves,
or a left move followed by only right moves. Note that moves of the first type will be right moves if and only if
$(-e_1)\cdots (-e_j) =1$, and hence moves of the
second or third types are left moves in this case. Note also that each move has the end result of switching the sign of the product of the $-e_i$.
We now consider three cases:

\emph{Case 1.} If the first move is of the first type, then the second move must be as well, in order to switch direction. Subsequent moves must all
have the same direction as the second, so they must alternate between type three moves (since the type one moves leave $(1,1)$ as the last term) and
type one moves. In this case, the continued fraction of $y$ is of the form $[(e_1, a_1), (e_2, a_2), \ldots , (e_j, a_j), (1, 1 + 2k)]$ or
$[(e_1, a_1), (e_2, a_2), \ldots , (e_j, a_j), (1, 1 + 2k), (1,1)]$.

\emph{Case 2.} If the first move is of second type, then the second move must be of third type, after which it must alternate between first
type and third type. Hence the continued fraction of $y$ is of the form
$[(e_1, a_1), (e_2, a_2), \ldots , (e_j, a_j), (-1, 1 + 2k)]$ or $[(e_1, a_1), (e_2, a_2), \ldots , (e_j, a_j), (-1, 1 + 2k), (1,1)]$.

\emph{Case 3.} If the first move is of third type, then the second move must be of second type, after which it must alternate between first
type and third type. Hence the continued fraction of $y$ is of the form $[(e_1, a_1), (e_2, a_2), \ldots , (e_{j-1}, a_{j-1}+2), (-1, 1), (1, 1 + 2k)]$ or
$[(e_1, a_1), (e_2, a_2), \ldots , (e_{j-1}, a_{j-1}+2), (-1, 1), (1, 1 + 2k), (1,1)]$. Note that in this case, we must have $a_j =1$.

In any case, what we need is the inequality $|Q_O(x) - Q_O(y)| \leq C\lambda^{-(a_1+\cdots+a_j)} \lambda^{-2k -1}$, and its consequence that since
$y \in \mathcal{X}_n$, then $$|Q_O(x) -Q_O(y)| \leq C \lambda^{-n}.$$

For the first two cases, this is an immediate consequence of the finite sum version of formula \eqref{eq6} for $Q_O$ and the possible continued fractions for $y$.
In these cases, the first $j$ terms of the continued fraction for $y$ coincide with those of $x$, causing the first $j$ terms of $Q_O(x)$ and $Q_O(y)$ to cancel,
leaving only one or two terms of order $\lambda^{-n}$. In the third case, we note that
\begin{equation*}
|Q_O(x) - Q_O(y)|= \lambda^{-(a_1+\cdots+a_{j-1})+1}|(1 -\lambda^{-1}) -(\lambda^{-2} + \lambda^{-3} - \lambda^{-4 -2k} + \lambda^{-4 -2k-1} )|
\end{equation*}
From the definition of $\lambda$ we have $1 - \lambda^{-1} - \lambda^{-2} -\lambda^{-3} =0$, so
\begin{equation*}
|Q_O(x) - Q_O(y)| \leq 2 \lambda^{-2k-2} \lambda^{-(a_1+\cdots+a_{j-1})+1} = (2\lambda)  \lambda^{-(a_1+\cdots+a_j)} \lambda^{-2k-1}.
\end{equation*}
Essentially, what was used in the third case is that $Q_O(x)$ does not depend on the representation of $x$. Although we have adopted a convention
that if the last $a_j =1$ then we require $e_j =1$, the formula for $Q_O$ gives the same results for $[(e_1, a_1), (e_2, a_2), \ldots, (e_{j-1}, a_{j-1}), (1, 1)]$
and the equivalent $[(e_1, a_1), (e_2, a_2), \ldots, (e_{j-1}, a_{j-1}+2), (-1, 1)]$, as a consequence of the definition of $\lambda$.

Finally, by increasing the constant $C$ by a factor of $\lambda$, we may remove our initial assumption that $y \in \mathcal{X}_n$, since given any two
adjacent elements of $\mathcal{Y}_n$, at least one of them must be in $\mathcal{X}_n$ or $\mathcal{X}_{n-1}$.
\end{proof}

We are now ready to prove Theorem $\ref{oddcont}$, in much the same manner as Theorem $\ref{evencont}$.

\begin{proof}[Proof of Theorem \ref{oddcont}]
Suppose $x < x'$ in $[0,1]$. Let $y = Q_O(x)$ and $y' = Q_O(x')$. Let $k$ be the least integer such that we have
$x \leq r \leq r' \leq x'$ for some $r, r' \in \mathcal{Y}_k$. The bound from Proposition $\ref{eq8}$ gives $x' -x \geq r' -r \geq {G^{-2k-4}}$ since $r$ and $r'$ have denominator at most $G^{k+2}$.
Since we have taken $k$ to be the least possible, there are at most $3$ elements of $\mathcal{Y}_k$ in the interval $[x, x']$, so $y' - y \leq 5 C \lambda^{-k}$.
Therefore
\begin{equation*}
y' -y \leq 5C\lambda^2(x'-x)^{\frac{\log \lambda}{2 \log G}}.
\end{equation*}

To see that this is best possible, consider $x = G -1 = [(1,1), (1,1), (1,1),\ldots]$ and its convergents. If
$x_n$ denotes the $n$th convergent of $x$, then $|x - x_n|$ is of the order $G^{-2n}$. On the other hand,
\begin{equation*}
|Q_O(x) - Q_O(x_n)| = \bigg| \sum_{k=n}^{\infty} (-\lambda)^{1-k} \bigg|
\end{equation*}
is of order $\lambda^{-n}$.
Since $|x - x_n|^{\frac{\log \lambda}{2 \log G}}$ is of order $(G^{-2n})^{\frac{\log \lambda}{2 \log G}} = \lambda^{-n}$,
we conclude that this is the best possible exponent.
\end{proof}

\begin{cor}
The map $Q_O$ is strictly increasing on $[0,1]$.
\end{cor}

\begin{proof}
It follows from Proposition \ref{Prop8} and
$\Q \cap [0,1]=\bigcup_{n\geq 1} \YY_n$ that $Q_O (\Q \cap [0,1])$ is dense in $[0,1]$.
Since $Q_O$ is continuous and it is non-decreasing by its very definition,
it follows that $Q_O$ is strictly increasing.
\end{proof}

\subsection{The linearization of the map $F_O$}

Consider the piecewise linear maps $\overline{F}_O, \overline{T}_O :[0,1]\rightarrow [0,1]$ defined by
\begin{equation*}
\begin{split}
\overline{F}_O (y) & =\begin{cases} \lambda^2 y & \mbox{\rm if $y\in [0,\lambda^{-2}]$} \\
\lambda(\lambda^2 y-1) & \mbox{\rm if $y\in [\lambda^{-2},1-\lambda^{-1} ]$} \\
\lambda (1-y) & \mbox{\rm if $y\in [1-\lambda^{-1},1],$}\end{cases}  \\
\overline{T}_O (y) & =\begin{cases} \lambda-\lambda^{2k-1} y & \mbox{\rm if $y\in (\frac{\lambda-1}{\lambda^{2k-1}},\frac{1}{\lambda^{2k-2}}),
\ k\geq 1$} \\
\lambda^{2k-1} y-\lambda  & \mbox{\rm if $y\in (\frac{1}{\lambda^{2k-2}},\frac{\lambda-1}{\lambda^{2k+1}}),\  k\geq 2.$}
\end{cases}
\end{split}
\end{equation*}

\begin{proposition}
The homeomorphism $Q_O$ of $[0,1]$ linearizes the maps $F_O$ and $T_O$ as follows:
\begin{equation*}
\mbox{\rm (i)} \   Q_O F_O Q_O^{-1} = \overline{F}_O, \qquad
\mbox{\rm (ii)} \  Q_O T_O Q_O^{-1} = \overline{T}_O.
\end{equation*}
\end{proposition}

\begin{proof}
Let $x=[(1,a_1),(e_1,a_2),(e_2,a_3),\ldots ]\in (0,1)$ and employ formula \eqref{eq6}.

(i) There are three cases to be considered:

\emph{Case 1.} $x\in (0,\frac{1}{3})$, where $(a_1,e_1)\notin \{ (3,-1),(1,1)\}$. Then we successively infer
\begin{equation*}
\begin{split}
Q_O (x) & =\frac{1}{\lambda^{a_1-1}} -\frac{e_1}{\lambda^{a_1+a_2-1}} +\frac{e_1e_2}{\lambda^{a_1+a_2+a_3-1}} -\cdots \in (0,\lambda^{-2}), \\
(\overline{F}_O Q_O)(x)  & =\lambda^2 Q_O (x)
= \frac{1}{\lambda^{a_1-3}}-\frac{e_1}{\lambda^{a_1+a_2-3}} +\frac{e_1e_2}{\lambda^{a_1+a_2+a_3-3}} -\cdots \\
& = Q_O ( [(1,a_1-2),(e_1,a_2),(e_2,a_3),\ldots ])
 = (Q_O F_O)(x) .
\end{split}
\end{equation*}

\emph{Case 2.} $x\in (\frac{1}{3},\frac{1}{2})$, so $a_1=3$, $e_1=-1$ and we have
\begin{equation*}
\begin{split}
Q_O (x) & = \frac{1}{\lambda^2} +\frac{1}{\lambda^{a_2+2}} -\frac{e_2}{\lambda^{a_2+a_3+2}}
+\frac{e_2 e_3}{\lambda^{a_2+a_3+a_4+2}} - \cdots \in (\lambda^{-2},1-\lambda^{-1}) , \\
(\overline{F}_O Q_O)(x) &  =\lambda (\lambda^2 Q_O(x)-1) =\frac{1}{\lambda^{a_2-1}}-\frac{e_2}{\lambda^{a_2+a_3-1}}
+\frac{e_2e_3}{\lambda^{a_2+a_3+a_4-1}} -\cdots \\
& = Q_Q ( [(1,a_2),(e_2,a_3),(e_3,a_4),\ldots ]) =(Q_O F_O)(x).
\end{split}
\end{equation*}

\emph{Case 3.} $x\in (\frac{1}{2},1)$, so $a_1=e_1=1$ and we have
\begin{equation*}
\begin{split}
Q_O(x) & =1-\frac{1}{\lambda^{a_2}} +\frac{e_2}{\lambda^{a_2+a_3}} -\frac{e_2e_3}{\lambda^{a_2+a_3+a_4}} + \cdots
\in (1-\lambda^{-1},1) , \\
(\overline{F}_O Q_O)(x) & =\lambda \big( 1-Q_O(x)\big) =\frac{1}{\lambda^{a_2-1}}-\frac{e_2}{\lambda^{a_2+a_3-1}} +
\frac{e_2e_3}{\lambda^{a_2+a_3+a_4-1}} -\cdots  =(Q_O F_O)(x).
\end{split}
\end{equation*}

(ii) According to formula \eqref{eq6} we have
\begin{equation*}
(Q_O T_O)(x) =\frac{1}{\lambda^{a_2-1}} -\frac{e_2}{\lambda^{a_2+a_3-1}} +\frac{e_2e_3}{\lambda^{a_2+a_3+a_4-1}} - \cdots .
\end{equation*}
Note also that for every $k\in \N$ we have
\begin{equation*}
Q_O \bigg( \frac{1}{2k-1}\bigg) =\frac{1}{\lambda^{2k-2}} \quad \mbox{\rm and} \quad
Q_O \bigg( \frac{1}{2k} \bigg) =\frac{1}{\lambda^{2k-2}}-\frac{1}{\lambda^{2k-1}}
=\frac{\lambda-1}{\lambda^{2k-1}} .
\end{equation*}

Two situations can occur:

\emph{Case 1.} $x\in (\frac{1}{2k},\frac{1}{2k-1})$, $k\geq 1$, so $a_1=2k-1$, $e_1=1$ and we have
\begin{equation*}
\begin{split}
Q_O \bigg( \frac{1}{2k}\bigg) & =\frac{\lambda-1}{\lambda^{2k-1}}   < Q_O (x) < Q_O \bigg(\frac{1}{2k-1}\bigg) =\frac{1}{\lambda^{2k-2}} ,\\
(\overline{T}_O Q_O )(x) & =\lambda -\lambda^{2k-1} Q_O (x) \\
& = \lambda-\lambda^{2k-1} \bigg( \frac{1}{\lambda^{2k-2}} -\frac{1}{\lambda^{2k-2+a_2}}
+\frac{e_2}{\lambda^{2k-2+a_2+a_3}} - \cdots \bigg)
= (Q_O T_O)(x).
\end{split}
\end{equation*}

\emph{Case 2.} $x\in (\frac{1}{2k-1},\frac{1}{2k-2})$, $k\geq 2$, so $a_1=2k-1$, $e_1=-1$ and we have
\begin{equation*}
\begin{split}
Q_O \bigg( \frac{1}{2k-1}\bigg) & =\frac{1}{\lambda^{2k-2}}   < Q_O (x) < Q_O \bigg(\frac{1}{2k-2}\bigg) =\frac{\lambda-1}{\lambda^{2k-3}}
=\frac{\lambda+1}{\lambda^{2k-1}} ,\\
(\overline{T}_O Q_O )(x) & = \lambda^{2k-1} Q_O (x) -\lambda \\
& = \lambda^{2k-1} \bigg( \frac{1}{\lambda^{2k-2}} +\frac{1}{\lambda^{2k-2+a_2}}
-\frac{e_2}{\lambda^{2k-2+a_2+a_3}} + \cdots \bigg) -\lambda
= (Q_O T_O)(x).
\end{split}
\end{equation*}
Alternatively, part (ii) can be deduced from a dynamical argument as in Proposition \ref{prop5}.
\end{proof}

\section*{Acknowledgments}
We are grateful to the referees for their valuable input that
contributed to a number of clarifications and improved the presentation of the paper.

The first author would like to acknowledge partial support during his
visits to IMAR Bucharest by a grant from Romanian Ministry of
Research and Innovation, CNCS-UEFISCDI, project PN-III-P4-ID-PCE-2016-0823, within PNCDI III.

\end{document}